\documentclass[11pt]{amsart}
\usepackage[english]{babel}
\usepackage[utf8]{inputenc} % accents dans le .tex (latin1 : Windows, utf8 : Linux) % utf8x
\usepackage[T1,OT1]{fontenc} %fontes accentu\'{e}es, c\'{e}sure des mots accentu\'{e}s	

\usepackage{amssymb} %ensembles, \leqslant, \boldsymbol{al{TEXTE}}
\usepackage{amsmath} %\DeclareMathOperator, \mod, pmatrix, vmatrix, \dfrac, \boxed, \iint, \text, \tag, \boldsymbol	mathcal{TEXTE}}, \text{}
\usepackage{amsthm} %\proof, \theoremstyle, \newtheorem
\usepackage{graphicx,subfigure}	
\usepackage{fullpage} % enlève les marges droite et gauche
\usepackage{epsfig,color}
\usepackage{enumerate}
\usepackage[table]{xcolor}
\usepackage{natbib}
\usepackage{enumitem}

\usepackage{float} %figures ici

\usepackage{tikz}
\usetikzlibrary{shapes}

\usepackage{pgfplots}
\usepackage{todonotes}

\definecolor{green_mermoz}{rgb}{0,0.5,0}
\usepackage{hyperref}

\hypersetup{colorlinks = true, linkcolor=green_mermoz, citecolor=orange}

\usepackage[noabbrev]{cleveref}
\creflabelformat{equation}{#2(#1)#3}
\crefname{thm}{Theorem}{Theorems}
\crefname{problem}{Problem}{Problems}
\crefname{prop}{Proposition}{Propositions}
\crefname{lem}{Lemma}{Lemmas}
\creflabelformat{problem}{#2(#1)#3}
\definecolor{k4}{rgb}{0.8,0.8,0.8}
\definecolor{k3}{rgb}{0.6,0.6,0.6}
\definecolor{k2}{rgb}{0.4,0.4,0.4}
\definecolor{k1}{rgb}{0.2,0.2,0.2}

\newcommand{\loc}{\mathrm{loc}}

\newcommand{\comega}{\Omega^{\mathsf{c}}}

\newcommand{\N}{\mathbb{N}}

\newcommand{\R}{\mathbb{R}}

\newcommand{\eps}{\varepsilon}
\newcommand{\f}{\dfrac}
\newcommand{\p}{\partial}

\newcommand{\abs}[1]{\left\lvert#1\right\rvert}
\newcommand{\norm}[1]{\left\lVert#1\right\rVert}
\newcommand{\nalph}[1]{\left\lVert#1\right\rVert_\alpha}
\newcommand{\Ea}{\mathcal{E}^\alpha}
\newcommand{\Eaz}{\mathcal{E}^\alpha_0}
\newcommand{\QV}{G_\e^V}
\newcommand{\dQVA}{dG_\e^{V_1}}
\newcommand{\dQVB}{dG_\e^{V_2}}
\newcommand{\dQVe}{dG_\e^{V_\e}}

\newcommand{\bz}{\overline{z}}

\newcommand{\Cgamma}{L_K(\gamma)}%\newcommand{\Zv}{Z_\mathrm{evol}}
\renewcommand{\epsilon}{\varepsilon}

\newcommand{\RK}{R_K}
\newcommand{\GK}{G_K}

\newcommand{\Gammam}{\Gamma^m_{\e}}
\newcommand{\GammaI}{\Gamma^I_{\e}}

\DeclareMathOperator{\Id}{Id}

\DeclareMathOperator{\I}{I}

\DeclareMathOperator{\lep}{\lambda_\varepsilon}
\DeclareMathOperator*{\argmin}{argmin}

\renewcommand{\H}{\mathcal{H}_\e}
\newcommand{\D}{\mathcal{D}_\e}
\renewcommand{\I}{\mathcal{I}}
\newcommand{\Je}{\mathcal{J}_\e}
\renewcommand{\S}{\mathcal{S}}

\renewcommand{\leq}{\leqslant}
\renewcommand{\geq}{\geqslant}

\renewcommand{\eps}{\varepsilon}

\renewcommand{\tilde}{\widetilde}

\newcommand{\m}{m}

\newcommand{\e}{\varepsilon}

\newcommand{\dint}{\displaystyle \int}
\newcommand{\vphia}{\varphi_\alpha}

\newcommand{\iep}[1]{I_\varepsilon (U_\varepsilon)(#1)}

\newtheorem{thm}{Theorem} %[section]
\newtheorem{lem}[thm]{Lemma} %[section]
\newtheorem{prop}[thm]{Proposition} %{prop}[thm]{Proposition}
\newtheorem{remark}[thm]{Remark} %{prop}[thm]{Proposition}
\newtheorem{cor}[thm]{Corollary}
\newtheorem{definition}[thm]{Definition}
\numberwithin{equation}{section}
\numberwithin{thm}{section}

\def\ds{\displaystyle}
\usepackage{epstopdf}

\title{Asymptotic analysis of a quantitative genetics model with nonlinear integral operator}

\author{Vincent Calvez}
\address{ICJ, UMR 5208 CNRS \& Universit\'{e} Claude Bernard Lyon 1, Lyon, France}
\email{vincent.calvez@math.cnrs.fr} 
\author{Jimmy Garnier} 
\address{LAMA, UMR 5127 CNRS \& Univ. Savoie Mont-Blanc, Chambéry, France}
\email{jimmy.garnier@univ-smb.fr}
\author{Florian Patout} 
\address{UMPA, UMR 5669 CNRS \& Ecole Normale Sup\'erieure de Lyon, Lyon, France} 
\email{florian.patout@ens-lyon.fr}
\date{\today}

\begin{document}

\maketitle

%\tableofcontents

%\newpage

%\listoffigures
 
%\newpage
\begin{abstract}
We study the asymptotic behavior of stationary solutions to a quantitative genetics model with trait-dependent mortality and a nonlinear integral reproduction operator.
Our asymptotic analysis encompasses the case when the deviation between the offspring and the mean parental trait is typically small. Under suitable regularity and growth conditions on the mortality rate, we prove existence and local uniqueness of a stationary profile that get concentrated around a critical point of the mortality rate, with a nearly Gaussian distribution having small variance. Our approach is based on perturbative analysis techniques that require to describe accurately  the correction to the Gaussian leading order profile. Our contribution extends previous results obtained with linear reproduction operator, but using an alternative methodology. 
\end{abstract}

\section{Introduction}%\label{sec:1}
We investigate solutions $(\lambda_\e,F_\e) \in \R\times L^1(\R^d)$ of the following stationary problem:
\begin{align}\label[problem]{eq:main_fep}\tag{$PF_\e$}
 \ \lambda_\e F_\e(z) + m(z) F_\e(z) =   \mathcal{B}_\e (F_\e)(z)\,, \quad z\in \R^d \, ,
\end{align}
where $\mathcal{B}_\e $ is the  following non linear, homogeneous integral operator associated to the infinitesimal model \cite{Fisher1918,weberinfinitesimal}:
\begin{equation}
\mathcal{B}_\e (f)(z) := \dfrac{1}{\eps^d \pi^\frac d2}   \iint_{\mathbb{R}^{2d}}  \exp\left[ -  \dfrac{1}{\eps^2} \left( z - \dfrac{ z_1 +   z_2}2 \right)^2 \right ]  f( z_1)\dfrac{  f(  z_2)}{ \int_{\R^d}   f(  z_2')\, d  z_2'}\, d  z_1 d  z_2.\label{eq:Beps}
\end{equation}

In the context of quantitative genetics, the variable $z$ denotes a multi-dimensional phenotypic trait, $F_\e(z)$ is the phenotypic distribution of the population and $m(z)$ is the (trait-dependent) mortality rate which results in the selection of the fittest individuals.

The mixing operator $\mathcal{B}_\e$ acts as a simple model for the inheritance of quantitative traits in a population with a sexual mode of reproduction. As formulated in \eqref{eq:Beps}, it is assumed that offspring traits are distributed normally around the mean of the parental traits $(z_1 + z_2)/2$, with a variance which remains constant accross generations, here $\eps^2/2$. 

We are interested in the asymptotic behaviour of the trait distribution $F_\eps$ as $\eps$ vanishes.

% where $z$ denotes a phenotypic trait variable, $F_\e(z)$ is the phenotypic distribution of the population, $m(z)$ is the (trait-dependent) mortality rate, and finally $\mathcal{B}_\e (f)$ is the  following non linear, homogeneous operator associated to the infinitesimal model \cite{Fisher1918,weberinfinitesimal}:
% \begin{equation}
% \mathcal{B}_\e (f)(z) := \dfrac{1}{\eps^d \pi^\frac d2}   \iint_{\mathbb{R}^{2d}}  \exp\left[ -  \dfrac{1}{\eps^2} \left( z - \dfrac{ z_1 +   z_2}2 \right)^2 \right ]  f( z_1)\dfrac{  f(  z_2)}{ \int_{\R^d}   f(  z_2')\, d  z_2'}\, d  z_1 d  z_2.\label{eq:Beps}
% \end{equation}
% This mixing operator $\mathcal{B}_\e$ acts as a simple model for the inheritance of quantitative traits in a population with a sexual mode of reproduction. As formulated in \eqref{eq:Beps}, it is assumed that offspring traits are distributed normally around the mean of the parental traits $(z_1 + z_2)/2$, with a variance which remains constant accross generations. Here, we are interested in the asymptotics of the trait distribution $F_\eps$ as the variance $\eps^2$ vanishes. 

This asymptotic regime was investigated thoroughly for various linear operators $\mathcal{B}_\e$ associated with asexual reproduction such as for instance the diffusion operator $F_\e(z) + \e^2 \Delta F_\e(z)$, or the convolution operator $\frac1\e K(\frac z\e)*F_\e(z)$ where $K$ is a probability kernel with unit variance, see \cite{diekmann2005,perthame-book,barlesperthamecontemporary,
mirrahimiconcentration,
lorz2011dirac} for the earliest investigations, see further \cite{meleard_singular_2015,
mirrahimi_singular_2018,bouin_thin_2018} for the case of a fractional diffusion operator (or similarly a fat-tailed kernel $K$), and see further \cite{mirrahimi_adaptation_2013,
mirrahimi_asymptotic_2015,
bouin_hamiltonjacobi_2015,
lam_integro-pde_2017,
gandon_hamiltonjacobi_2017,
mirrahimi_hamiltonjacobi_2017, mirrahimi_evolution_2018,calvez_non-local_2018} for the interplay between evolutionary dynamics and a spatial structure. In the linear case, the asymptotic analysis usually leads to a Hamilton-Jacobi equation for the Hopf-Cole transform $U_\e = -\e \log F_\e$. This yields an original problem with non-negativity constraint that requires a careful well-posedness analysis \cite{mirrahimi_class_2015,
calvez_uniqueness_2018}. 

Much less is known about the non linear equation \eqref{eq:main_fep}, although this model  is widely used in theoretical evolutionary biology to describe sexual reproduction, see {\em e.g.} \cite{slatkin1970selection,
roughgarden1972evolution,
slatkin_niche_1976,
bulmerbook,menormal,
tufto_quantitative_2000,	
barfield_evolution_2011, huisman2012comparison,
cotto-ronce,
weberinfinitesimal,turelli_comm}.  
%and the nonlinear operator $\mathcal{B}_\e (f)$ acts as a simple model for the inheritance of quantitative traits in a population with a sexual mode of reproduction. As formulated in \eqref{eq:Beps}, it is assumed that offspring traits are distributed normally around the mean of the parental traits $(z_1 + z_2)/2$, with a variance $\epsilon^2$ which remains constant accross generations.
%The underlying assumption is that, when two individuals with traits $z_1,z_2$ mate, the trait of the offspring is distributed normally  around the mean trait of the parents $(z_1 + z_2)/2$ \cite[see][ for biological details]{bulmerbook,burgerbook,doebeli2007multimodal}. The positive parameter $\e$ tunes the typical deviation  size of the offspring trait from the mean of the parents traits. 

From a mathematical viewpoint, the model~\eqref{eq:Beps} received recent attention in the field of probability theory \cite{weberinfinitesimal} and integro-differential equations~\cite{raoulmirrahimiinfinitesimal,Raoul}. 
In the latter couple of articles, a scaling different from \eqref{eq:Beps} is studied: the variance is of order one, but there is a large reproduction rate that enforces the relaxation of the phenotypic distribution towards a Gaussian local equilibrium. Macroscopic equations are rigorously derived in \cite{Raoul}, in the case of an additional spatial structure, in the spirit of hydrodynamic limits for kinetic equations. 

In a different context, a similar collisional operator as  $\mathcal{B}_\e$ \eqref{eq:Beps} was introduced in the modelling of self-propelled particles with alignment interactions, see for instance \cite{bertin_boltzmann_2006,degond_local_2014}. When two particles interact they tend to align with the mean velocity, with some possible noise. However, there are some discrepancies with our case study, since the operator is not conservative in our case, by definition of a reproduction operator. Moreover, it is normalized by the total mass of the phenotypic distribution: $\int  f(  z_2')\, d  z_2'$. The rationale behind this choice is that during the mating process, the first parent chooses the trait of its partner depending on its frequency in the population. This is the neutral case without any assumption about assortative mating. Moreover, this dependency upon the frequency rather than the density discards any small population effects that could arise from a quadratic collisional operator. Such homogeneity of degree 
one is a key ingredient in our analysis.

The problem \eqref{eq:main_fep} is equivalent to the existence of special solutions  of the form $\exp (\lambda_\e t) F_\e(z)$, for  the following non-linear but one-homogeneous equation which will be the subject of future work:
\begin{equation}\label{eq:time marching}
\p_t f(t,z) + m(z)f(t,z) = \mathcal{B}_\e(f)(t,z)\,, \quad t>0\,,\;  z\in \R^d .
\end{equation}
Alternatively speaking, the problem \eqref{eq:main_fep} expresses the balance between selection via trait-dependent mortality $m(z)$, and the generation of diversity through reproduction $\mathcal{B}_\epsilon$. The scalar $\lambda_\e$ is analogous to the principal eigenvalue of the operator $\mathcal{B}_\epsilon - m$. However, it might not be unique, as in the Krein-Rutman theory, see Corollary \ref{cor:uniqueness}. It measures the global fitness of the population: the population  grows exponentially fast $\lambda_\e>0$ when the reproduction term $\mathcal{B}_\epsilon$ dominates, while it declines exponentially fast $\lambda_\e<0$ when the mortality $m$ out-competes the reproduction.  

This preliminary analysis on the stationary profile paves the way for a systematic analysis of various quantitative genetics models, including time marching problems and the combination of multiple effects (spatial structure, aging of the population etc).

Our work is inspired by similar asymptotics in the case of linear operator $\mathcal{B}_\epsilon$, see the seminal work by \cite{diekmann2005} and references cited above. Accordingly, our goal is to analyze problem \eqref{eq:main_fep} in the limit of vanishing variance $\e^2\to 0$. Since there is few diversity generated in this asymptotic regime, we expect that the  variance of the distribution solution $F_\epsilon$ vanishes as well. Actually, there is strong evidence that the leading order profile of $F_\epsilon$ is a Gaussian distribution with variance $\e^2$. As a matter of fact, any Gaussian distribution with variance $\e^2$ is invariant by the infinitesimal operator $\mathcal{B}_\e$ in the absence of selection ($m \equiv 0$, $\lambda_\e = 1$) \cite{menormal,raoulmirrahimiinfinitesimal}. This motivates the following decomposition of the solution: \begin{equation}\label{eq:def_Ueps}
   F_\e(z) = \dfrac{1}{(2\pi)^\frac d2 \e^d} \exp\left( -\dfrac{(z -z_0 )^2}{2\epsilon^2}  - U_\epsilon(z) \right). 
\end{equation}
The latter \eqref{eq:def_Ueps}  is similar to the Hopf-Cole transform used in the asymptotic analysis of adaptative evolutionary dynamics in asexual populations.  In our case $U_\e$ is a corrector term that measures the deviation from the leading Gaussian distribution of variance $\e^2$. %[\footnote{Attention la variance c'est $\e^2$ pas $\e$}].  
Our analysis reveals that selection determines the center of the distribution $z_0$, as expected, and also reshapes the distribution $F_\e$ via the corrector $U_\e$. 

The operator $\mathcal{B}_\e$ is invariant by translation. Up to a translation of $m$, we may assume that the leading order Gaussian distribution is centered at the origin, {\em i.e.} $z_0 = 0$. Next, up to a change of $\lambda_\e  \leftarrow \lambda_\e + m(0)$, we may assume that $m(0) = 0$. Note that we may also assume $U_\e(0) = 0$ without loss of generality, as the original problem is homogeneous. 

Plugging the transformation of  \eqref{eq:def_Ueps} into \eqref{eq:main_fep} yields the following equivalent  problem for $U_\e$ :
\begin{align}
\label[problem]{eq:PUeps}\tag{$PU_\e$} \lambda_\e + m(z) & = I_\e(U_\e)(z) \exp \left( U_\e(z)-2U_\e\left(\frac{z}{2} \right) + U_\e(0) \right),  \quad z \in \R^d.
\end{align}
The residual term from the integral contribution is the following non-local term $I_\e(U_\epsilon)$, see Section \ref{sec:reformulation1} for the details of the derivation:
\begin{multline}\label{eq:def_Ieps_int}
I_\e(U_\e)(z) \\ =  \dfrac{\ds  \iint_{\R^{2d}} \exp \left[-\dfrac{1}{2}y_1 \cdot y_2 - \dfrac{3}{4}\left (\abs{y_1}^2+\abs{y_2}^2\right ) +2 U_\e\left( \frac{z}{2} \right) - U_\e\left( \frac{z}{2}+\epsilon y_1 \right) - U_\e\left(\frac{z}{2} +\epsilon y_2\right)\right] d y_1 d y_2  }
                                 {\ds   \pi^{d/2} \int_{\R^d} \exp\left[ - \frac12 |y|^2 +  U_\e(0)- U_\e(\epsilon y) \right] d y}. 
\end{multline}
%where $N$ stands for the normalized Gaussian distribution.  
This decomposition appears to be relevant because a formal computation shows that $I_\e(U_\e)\to 1$ as $\e\to 0$. Establishing uniform convergence is actually a cornerstone of our analysis. Thus for small $\epsilon,$ the problem \eqref{eq:PUeps} is presumably  close to the following  corrector equation, obtained formally at $\epsilon =0$: 
% \vc{The reason for decomposing the problem as in \eqref{eq:PUeps} is that $I_\e(U_\e)\to 1$ uniformly as $\e\to 0$ in our analysis, so that \eqref{eq:PUeps} is close to the following limiting corrector equation that was identified in \cite{}:} 
\begin{equation}\label[problem]{eq:PU0}\tag{$PU_0$}
\lambda_0 +m(z) =  \exp \left( U_0(z)-2U_0\left(\frac{z}{2}\right) + U_0(0) \right), \quad z \in \R^d.
\end{equation}
Interestingly, this finite difference equation admits explicit solutions by means of an infinite series:
\begin{equation*}
U_0(z) = \gamma_0\cdot z + \sum_{k\geq 0} 2^k \log \left( \lambda_0 + m(2^{-k}z )\right)\, ,
\end{equation*} 
 However, two difficulty remains: identify {\em(i)} the linear part $\gamma_0\in\R^d$ and {\em(ii)} the unknown $\lambda_0\in\R$.
On the one hand, the linear part $\gamma_0$ cannot be recovered from \eqref{eq:PU0} because linear contributions cancel in the right-hand-side of \eqref{eq:PU0}. Thus, identifying the coefficient $\gamma_0$ will be a milestone of our analysis. On the other hand, two important conditions must be fulfilled to guarantee that the series above converges, namely: 
\[ \lambda_0 + m(0) = 1\,,\quad  \text{and} \quad \p_z m(0) = 0 .\] 
The latter is a constraint on the possible translations that can be operated: the origin must be located at a critical point of $m$. The former prescribes the value of $\lambda_0$ accordingly. These two conditions are necessary conditions for the resolvability of problem \eqref{eq:PU0}. Indeed, evaluating \eqref{eq:PU0} at $z=0$, we get the first identity. Next, differentiating and evaluating again at $z = 0$, we get the second identity.

In the sequel we make this formal discussion rigourous, following a perturbative approach for $\e$ small enough. Before stating our main result, we need to prescribe the appropriate functional space for the corrector $U_\e$.

\begin{definition}[Functional space for $U_\e$]\label{def:Ea}$ $\\
For any positive parameter 	$\alpha \leq 2/5$, we define the functional space \begin{equation*}
\Ea =\left\{ 
u \in \mathcal{C}^3(\R^d)\::\: u(0)=0,  \text{ and } \left|
\begin{array}{rr} \ds  \abs{ D u(z)} &  \\
\ds  \left( 1+|z| \right) ^\alpha \abs{ D^2 u(z) }& \\
\ds  \left( 1+|z| \right) ^\alpha \abs{ D^3 u(z) }& \\
\end{array} \right. \in L^\infty(\R^d)
\right\},
\end{equation*}
equipped with the norm 
\begin{equation}\label{eq:normalpha}
\nalph{u} = \max \left(   \sup_{z \in \R^d} \abs{D u(z) } , \  \sup_{z \in \R^d} \   (1+\abs{z})^\alpha \left\{\abs{   D^2 u(z) } , \abs{ D^3 u(z) } \right\} \right). 
%\vc{\nalph{u} = \max \left( |D u(0)|,  \sup_{z \in \R^d} \abs{D u(z) - Du(0)} , \  \sup_{z \in \R^d} \   (1+\abs{z})^\alpha \left\{\abs{   D^2 u(z) } , \abs{ D^3 u(z) } \right\} \right).} 
\end{equation}
\end{definition}
For any bounded set $K$ of $\Ea$, we use the notation $\nalph{K} = \sup_{u\in K} \|u\|_\alpha$. Occasionally we use the notation $\vphia$ for the weight function $\vphia(z) = (1+\abs{z})^\alpha$. Although $2/5$ is not the critical threshold, it happens that the exponent $\alpha$ cannot be taken too large  in our approach. We set implicitly $\alpha = 2/5$ in the following results, however we leave it as a parameter to emphasize its role in the analysis, and to pinpoint the apparition of the threshold. Note that $\alpha>0$ is required in our approach, as one constant collapses in the limit $\alpha\to 0$ (see estimate \eqref{eq:alpha collapses} below).

Then, we detail the assumptions on the selection function $m$.
\begin{definition}[Assumptions on $m$]\label{def m}$ $\\
The function $m$ is a $\mathcal{C}^3(\R^d)$ function, bounded below, that admits a \emph{local} %[\footnote{j'ai envie de mettre 'local' en gros-gras-italique, un truc avec des étoiles qui brillent, ou alors un GIF animé}] 
 non-degenerate minimum at $0$ such that $m(0)=0$, and there exists $\mu_0>0$ such that $D^2 m(0) \geq \mu_0 \Id$ in the sense of symmetric matrices.
Furthermore we suppose that $(\forall z)\; 1+m(z) >  0$ and
\begin{align}\label{eq:log m}
(1+\abs{z})^\alpha \dfrac{D^k m(z)}{1+m(z)} \in L^\infty(\R^d)\, , \quad \text{for\,  $k = 1,2,3$}\, .
\end{align}
\end{definition}
\begin{remark}\label{rem:SM} Our result is insensitive to the sign of the local extremum. Indeed, one can replace  the hypothesis that $m$ admits a "local non degenerate  \emph{minimum}" at $0$ with a "local non degenerate   \emph{maximum}" at 0, and that there exists $\mu_0 <0$ such that $ D^2 m(0) \leq \mu_0 \Id$. However, we leave our main assumption as in Definition \ref{def m} as it is the most natural one from the point of view of stability analysis for the time-marching problem \eqref{eq:time marching}.
\end{remark}
The   condition \eqref{eq:log m} is clearly verified if $m$ is a polynomial function. It would be tempting to write, in short, that $\log(1+m)\in \Ea$, which is indeed a consequence of \eqref{eq:log m}. However, the latter condition also contains the decay of the first order derivative $D\log(1+m)$ with rate $|z|^{-\alpha}$, which is not  contained in the definition of $\Ea$ \eqref{eq:normalpha} for good reasons.

We also introduce the subset $\Eaz$ :  
\begin{equation}\label{eq:E0}
\Eaz = \left\{ v \in \Ea\::\: D v(0)=0,  \ D^2 v(0) \geq   D^2 m(0) \geq \mu_0 \Id \right\},
\end{equation}
Then, our assumption on $m$ in fact  guarantees	 that
\begin{align}\label{cond logm inEaz}
\log(1+m) \in \Eaz.
\end{align}

The main result of this article is the following theorem :
\begin{thm}[Existence and convergence]\label{convergence PUeps}$ $ 
\begin{itemize}
\item[(i)] There exist $K_0$ a  ball of $\Ea$,  and $\e_0 $ a positive constant, such that for any $\e \leq \e_0$,  the  problem \eqref{eq:PUeps} admits a unique solution $(\lambda_\e, U_\e) \in \R  \times K_0$.
\item[(ii)] The family $(\lambda_\epsilon,U_ \epsilon)_{\e}$ converges to $(\lambda_0,U_0)$ as $\epsilon\to0$,  with 
\begin{align}
\label{eq:def_lambda0}   \lambda_0  &  =  1, \\
\label{eq:def_U0}  \ds  U_0(z)  & = \gamma_0 \cdot z + V_0(z) ,
\end{align}
where
\begin{align}
\label{eq:gamma0value} \gamma_0  = \left\{ \begin{array}{ll}  \dfrac{ \partial^3_z m(0)}{ 2 \partial ^2_z m(0)}\,, \quad\text{if}\quad d=1 \medskip\\
\dfrac{1}{2}  \left (D^2 m(0)\right )^{-1}  D ( \Delta m)(0)\,, \quad\text{if}\quad d>1
 \end{array} \quad\text{and}\quad  \right. V_0   = \sum_{k\geq 0}  2^k \log \left( 1 + m(2^{-k}z )\right) . 
\end{align}
%where the convergence $U_\e \to U_0$ holds locally uniformly up to the second derivative. 
%which is a solution to \eqref{eq:PU0} in $\R \times \Ea$.
Moreover, the convergence $U_\e \to U_0$ is locally uniform up to the second derivative. 
\end{itemize}
\end{thm}

An immediate remark is that the regularity required by \eqref{eq:log m}, and particularly the $\mathcal C^3$ regularity of $m$, is consistent with formula \eqref{eq:def_U0} which involves the pointwise value of third derivatives of $m$. Alternatively speaking we think that our result is close to optimal in terms of regularity.

It is important to notice that our result holds true for \emph{any} local mimimum $z_0$ such that 
\begin{equation}\label{comp cond}
m(z_0) <  1 + \inf m.
\end{equation}
One should define the functional spaces $\Ea$ and $\Eaz$ accordingly (and particularly replace the conditions $u(0) = 0$ and $Du(0) = 0$ by the conditions $u(z_0) = 0$ and $Du(z_0) = 0$), and then adapt \eqref{eq:def_lambda0}--\eqref{eq:def_U0} as follows, for the one-dimensional case:
\begin{align}
& \nonumber    \lambda_0  = 1- m(z_0),\\
& \label{eq:defU0trans}\ds U_0(z_0+h)   =  \gamma_0 \cdot h  + \sum_{k\geq 0} 2^k \log \left( 1 +m(2^{-k}(z_0+h))-m(z_0)\right),
\end{align}
where $\gamma_0$ is defined by the same formula as in \eqref{eq:gamma0value} but evaluated at $z_0$. Immediately, one sees that the compatibility condition \eqref{comp cond} is necessary to have the positivity of the term inside the $\log$ in \eqref{eq:defU0trans}.
As a consequence, we have:
\begin{cor}[Lack of uniqueness]\label{cor:uniqueness}$ $\\
If the selection function $m$ has at least two different local non-degenerate minima that verify the compatibility condition \eqref{comp cond}, there exists at least two pairs $(\lambda_\e,F_\e)$ solutions of  problem \eqref{eq:main_fep} for $\e$ small enough. 
\end{cor}

\begin{figure}
\begin{center}
\includegraphics[height=.33\linewidth]{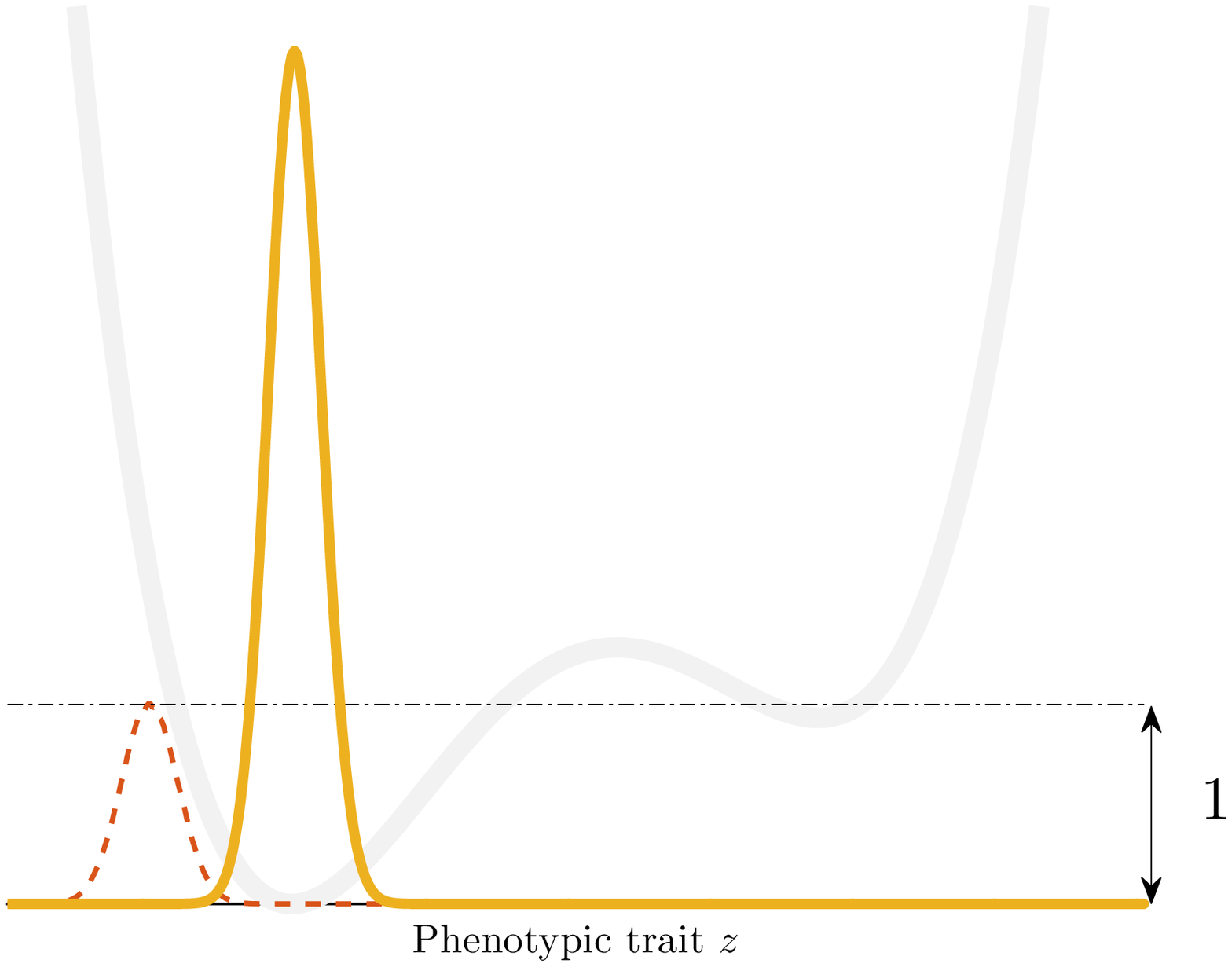} \quad
\includegraphics[height=.33\linewidth]{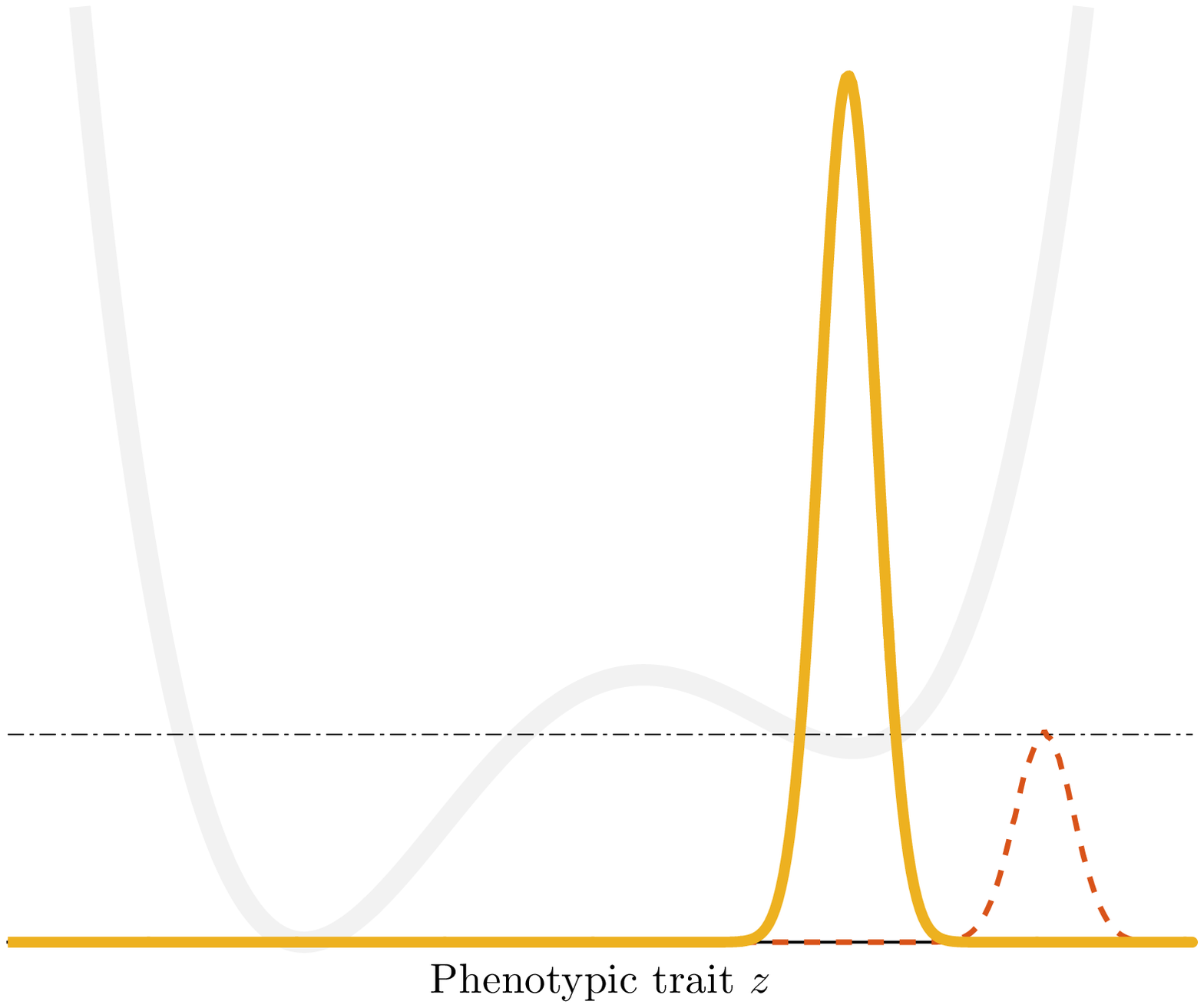} 
\caption{Numerical simulations of the stationary problem \eqref{eq:main_fep} with $\eps = 0.1$ in an asymmetric double-well mortality rate (grey line). The numerical equilibrium is in yellow plain line.  The only difference between the two simulations is the initial data (red dashed line). The simulations illustrate the lack of uniqueness for problem \eqref{eq:main_fep}.}
\end{center}\label{fig:doublewell}
\end{figure}

We  performed numerical simulations to illustrate this phenomenon (see Figure \ref{fig:doublewell}). The function $m$ is an asymmetric double well function. We solved the time marching problem \eqref{eq:time marching}   but on the renormalized density $F_\e/\int F_\e$ in order to catch a stationary profile. We clearly observed the co-existence of two equilibria for the same set of parameters, that were obtained for two different initializations of the scheme. 
However, let us mention that the question of  uniqueness in the case of a convex selection function $m$ is an open question,  to the extent of your knowledge.

This result is in contrast with analogous eigenvalue problems where $\mathcal{B}_\e$ is replaced with a linear operator, say $F_\e + \e^2 \Delta F_\e$  as in various quantitative genetics models with asexual mode of reproduction, see {\em e.g.} \cite{mirrahimiconcentration} and references mentioned above, or in the semi-classical analysis of the Schr\"odinger equation, see {\em e.g.} \cite{Sjostrand}. In the linear case, $\lambda_\e\in \R$ and $F_\e\geq 0, F_\e\not\equiv 0$ are uniquely determined (up to a multiplicative constant for $F_\e$) under mild assumptions on the potential $m$. This is the signature that $\mathcal{B}_\e$ \eqref{eq:Beps} is genuinely non-linear and non-monotone, so that possible extensions of the Krein-Rutman theorem for one-homogeneous operators, as in \cite{nonlinearkreinrutman}, are not applicable.

The existence  part $(i)$ was already investigated  in \cite{bourgeronspec} using the Schauder fixed point theorem and loose variance estimates. But the approach was not designed  to catch the asymptotic regime $\e\to 0$.
The current methodology gives much more precise information on the behavior of the solutions of the problem \eqref{eq:main_fep} in the regime of vanishing variance.

Theorem 1 provides a rigorous background for the connection between problem \eqref{eq:PUeps} and problem \eqref{eq:PU0} in a perturbative setting. It justifies that the problem \eqref{eq:main_fep} is well approximated by the    solution $(\lambda_0,U_0)$ of the  problem \eqref{eq:PU0}. Quite surprisingly, the value $\gamma_0$ of the linear part of the corrector function 
$U_0$ is resolved during the asymptotic analysis although it cannot be obtained readily from problem \eqref{eq:PU0} as mentioned above. It coincides with the heuristics of \cite{main} where the same coefficient was obtained by studying the formal expansion up to the next order in $\e^2$: $U_\e = U_0 + \e^2 U_1 + o(\e^2)$, and by identifying the equation on $U_1$ in which the value of $\gamma_0$ appears as another compatibility condition. Here the value of $\gamma_0$ is obtained directly as a by-product of the perturbative analysis.

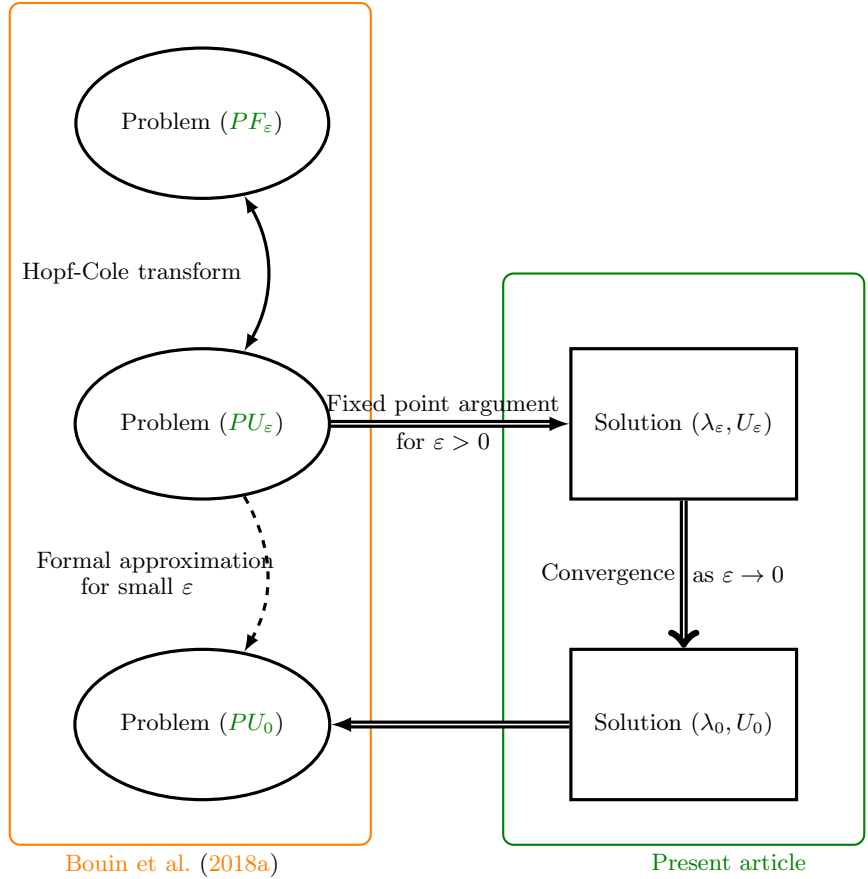
\begin{figure}
\begin{center}
\begin{tikzpicture}[scale=0.8]
\begin{footnotesize}
\draw[orange, rounded corners, thick] (-3.2,-7) rectangle (2.8,7) ;
\draw[green_mermoz, rounded corners, thick] (5,2.5) rectangle (11,-7) ;

\draw (-0.5,-7) node[below] {\cite{main} };
\draw (8.75,-7) node[below] {\textcolor{green_mermoz}{Present article}};

\node[draw, ellipse, minimum height=2cm, minimum width=3cm, very thick] (A) at (0,5) { Problem \eqref{eq:main_fep} };
\node[draw, ellipse, minimum height=2 cm, minimum width=3cm, very thick] (B) at (0,0) {Problem \eqref{eq:PUeps}};
\node[draw, ellipse, minimum height=2 cm, minimum width=3cm, very thick] (C) at (0,-5) {Problem \eqref{eq:PU0}};
\node[draw, minimum height=2 cm, minimum width=3cm, very thick] (E) at (8,-5) {Solution $(\lambda_0,U_0)$};
\node[draw, minimum height=2 cm, minimum width=3cm, very thick] (D) at (8,0) {Solution $(\lep,U_\e)$};

\draw[<->,>=latex, very thick] (A)to [bend left] (B);
\draw[->,>=latex, very thick, dashed] (B)to [bend left] (C);
\draw[->,>=latex, very thick, double] (B)to  (D);
\draw[->, very thick, double] (D)to  (E);
\draw[->,>= latex, very thick, double] (E)to  (C);

\draw (.8,2.5) node[left] {Hopf-Cole  transform};
\draw (0,-2.7) node[left] {for small $\e$};
\draw (-2.9,-2.3) node[right] {Formal approximation};

\draw (4,0) node[above] {Fixed point argument};
\draw (4,0) node[below] {for $\e > 0$};
\draw (8,-2.5) node[left] {Convergence};
\draw (8,-2.5) node[right] {as $\e \to  0$};

\end{footnotesize}
\end{tikzpicture}
\end{center}
\caption{Scope of our paper compared to precedent work	}
\label{fig:scheme}
\end{figure}

%In our study we construct, for small $\epsilon>0,$ a  solution $U_\e$ to problem \eqref{eq:PUeps} associated with strong regularity and smallness estimates. This method is commanded by the very rigid structure of the infinitesimal operator $\mathcal{B}_\epsilon$. This particular solution converges as $\epsilon\to 0$ towards a solution of the functional problem \eqref{eq:PU0}. 

As mentioned above, 
our approach is very much  inspired, yet different to most of the current literature about asymptotic analysis of asexual models, where the limiting problem is a Hamilton-Jacobi equation, see \cite{perthame-book} for a comprehensive introduction, and references above. To draw a parallel with our problem, let us consider the case where $\mathcal{B}_\epsilon(f)$ is replaced with the (linear) convolution operator $  K_\e*f$, where the kernel has the scaling property $  K_\e = \frac1\e   K\left ( \frac{\cdot}{\e}\right )$, and $K$ is a probability distribution kernel. There, the small parameter $\e$ measures the typical size of the deviation between the offspring trait and the sole parental trait. In this context, it is natural to introduce the Hopf-Cole transform $U_\e = - \e \log F_\e$. Then, the problem is equivalent to the asymptotic analysis of the following equation as $\e\to0$:
\begin{align}\label{eq:asex}
\lambda_\e + m(z) =    \int_{\R^d}   K\left(y\right) \exp\left( \dfrac{U_\e(z)-U_\e(z-\eps y)}\eps\right)\, dy,
 \end{align}
For this model, it is known  that      $U_\e$ converges towards the viscosity solution of a Hamilton-Jacobi equation \cite{mirrahimiconcentration}:
\begin{equation}\label{eq:asex2}
\lambda_0 +  m(z) =  H(D U_0(z)) =   \int_{\R^d}   K\left(y\right) \exp\left(   D U_0(z)\cdot y \right)\, dy\, . 
\end{equation}
Note that the limiting equation on $U_0$ \eqref{eq:asex2} can be derived formally from \eqref{eq:asex} by a first order Taylor expansion on $U_\e$.

There are two  noticeable discrepancies between the asexual case \eqref{eq:asex}--\eqref{eq:asex2} and our problem involving the infintesimal model with small variance. Firstly, $\e$ plays a similar role in both cases, {\em i.e.} measuring typical deviations between offspring and parental traits. However, the appropriate normalization differs by a factor $\e$: it is $-\e \log F_\e$ in the asexual case, whereas it is $-\e^2 \log F_\e$ in our context, see \eqref{eq:def_Ueps}. This scaling difference is the signature of major differences between the two problems (asexual vs. sexual). Secondly, the two limiting problems \eqref{eq:asex2} and \eqref{eq:PU0} have completely different natures: a Hamilton-Jacobi PDE in the asexual case, vs. a finite difference equation in the sexual case. Moreover, due to the lack of a comparison principle in the original problem \eqref{eq:main_fep}, we could not envision a similar notion of viscosity solutions for \eqref{eq:PU0}. Instead, we use rigid contraction properties and a 
suitable perturbative analysis to construct a unique strong solution near the limiting problem, as depicted in Figure \ref{fig:scheme}.
 
\cite{raoulmirrahimiinfinitesimal} observed that the infinitesimal operator $\mathcal{B}_\e$ alone enjoys a uniform contraction property with respect to the quadratic Wasserstein distance, with a factor of contraction $1/2$. Recently, this was used by \cite{magal-raoul} to perform a hydrodynamic limit in a different regime than the one under consideration here, see also \cite{Raoul}. However, the combination of $\mathcal{B}_\e$ with a zeroth-order heterogeneous mortality $m(z)$ seems to destroy this nice structure (details not shown).

The next section is devoted to the reformulation of problem \eqref{eq:PUeps} into a fixed point problem, introducing a set of notation and the strategy to prove \Cref{convergence PUeps}. The organization of the paper is postponed to the end of the next Section.

Up until the last part of the article we implicitly work in dimension $d=1$, for the readers' convenience. In \cref{sec large dim} we  pinpoint the few elements of the proof that are specific to the one-dimensional case and give an extension to the higher-dimensional case in order to complete the  proof of \Cref{convergence PUeps}. 

\subsection*{Acknowledgement.} The authors are grateful to Laure Saint-Raymond for stimulating discussions at the early stage of this work. They are thankful to Sepideh Mirrahimi for pointing out the extension of the result to local maxima of the selection function, see Remark \ref{rem:SM}.
 Part of this work was completed when VC was on temporary leave to the PIMS (UMI CNRS 3069) at the University of British Columbia.  This project has received funding from the European Research Council (ERC) under the European Union’s Horizon 2020 research and innovation programme (grant agreement No 639638) and from the French National Research Agency with the project NONLOCAL (ANR-14-CE25-0013) and GLOBNETS (ANR-16-CE02-0009). 

\section{Reformulation of the problem as a fixed point}\label{sec reformulation}
\subsection{Looking for problem $(PU_\e)$}
\label{sec:reformulation1}
The equivalence between problem \eqref{eq:main_fep} and problem \eqref{eq:PUeps} through the transform \eqref{eq:def_Ueps} is not immediate. It is detailed in \cite{main}, but we recall here the key steps for the sake of completeness. Plugging \eqref{eq:def_Ueps} into problem \eqref{eq:main_fep}  yields, with the notation $q(z)=\frac{z^2}{2}$:
\begin{multline*}%\label{eq:before limit}
\lambda_\e + \m(z) \\
= \f{  \displaystyle \iint_{\R^2} \exp\left[-\dfrac1{\eps^2} \left(   2 q\left(z-\dfrac{z_1+z_2}{2}\right) + q(z_1) + q(z_2)  - q(z) \right) - U_\e(z_1) - U_\e(z_2) + U_\e(z) \right] d z_1 d z_2}{\displaystyle  
\eps \sqrt{\pi}  \int_\R \exp\left(-\dfrac{q(z')}{\eps^2} - U_\e(z') \right) d z'}\, . 
\end{multline*}

When $\epsilon\to 0$, we expect the numerator integral to concentrate around the minimum of the principal term that is :
 \[ \argmin_{(z_1,z_2)} \left[ 2q \left(z-\frac{z_1+z_2}{2} \right) + q(z_1) + q(z_2) - q(z) \right] = \left(\frac{z}{2}, \frac{z}{2}\right).\]
We introduce the notation 
\[ \overline{z} = \dfrac{z}{2}  .\]
Using the change of variable $(z_1,z_2) = (\bz + \e y_1, \bz + \e y_2) $, we obtain the following equation :
\begin{equation}\label{eq:before limit 2}
\lambda_\e + \m(z) = \f{  \displaystyle \iint_{\R^2} \exp\left(-Q(y_1,y_2) - U_\e(\bz+\e y_1) - U_\e(\bz+\e y_2) + U_\e(z) \right) d y_1 d y_2}{\ds  
\sqrt{\pi}    \int_\R  \exp\left( - {y^2}/{2} - U_\e(\e y) \right) d y}\, , 
\end{equation}
where 
\begin{equation*}
\dfrac{1}{\e^2} \left[ 2q \left(z-\frac{z_1+z_2}{2} \right) + q(z_1) + q(z_2) - q(z) \right] = \dfrac{1}{2}y_1y_2 + \dfrac{3}{4}(y_1^2+y_2^2 ) = Q(y_1,y_2).
\end{equation*} 
\begin{definition}$ $\\
We denote by $Q$ the following quadratic form :
\begin{align*}%\label{def Q}
Q(y_1,y_2) = \dfrac{1}{2}y_1y_2 + \dfrac{3}{4}(y_1^2+y_2^2).
\end{align*}
\end{definition}
It is the residual quadratic form after our change of variable. We notice that $\frac{1}{\sqrt{2}\pi}\exp( -Q )$   is the density of a bivariate normal random variable with covariance matrix
\begin{equation}\label{eq:covariance}
\Sigma =\frac14 \begin{pmatrix}
3 & -1  \\ -1& 3
\end{pmatrix}.
\end{equation}
At the denominator of \eqref{eq:before limit 2} naturally arises $N$ the density function of a $\mathcal{N}(0,1)$ random variable.

Finally, \eqref{eq:before limit 2} is equivalent to problem \eqref{eq:PUeps}:
\begin{align*}
\lambda_\e + m(z) & = I_\e(U_\e)(z) \exp \left( U_\e(z)-2U_\e\left(\bz\right) + U_\e(0) \right) ,
\end{align*}
simply by conjuring $2 U_\e(z/2)$ at the numerator and $U_\e(0)$ at the denominator, resulting into the defintion of the remainder $I_\e(U_\e)$ \eqref{eq:def_Ieps_int} that will be controlled uniformly close to 1 in all our analysis. 

In the next section we explain how we reformulate the problem \eqref{eq:PUeps} into a fixed point argument in order to use a Banach-Picard fixed point theorem which prove our results rigorously.
\subsection{Some auxiliary functionals and the fixed point mapping}\label{eq:formal equivalence}
This section is devoted to the derivation of an alternative formulation for problem \eqref{eq:PUeps}.
Let $(\lambda_\epsilon, U_\epsilon)$ be a solution of problem \eqref{eq:PUeps} in $\R \times \Ea $.

The first step is to dissociate the study of $\lep$ and $U_\e$. We first evaluate the problem \eqref{eq:PUeps} at $z=0$. It yields the following  condition on $\lambda_\e$, since $m(0)=0$:
\begin{equation}\label{eq:nec_lep}
\lep = \iep{0}.
\end{equation}
%Comparing problem \eqref{eq:main_fep} and problem \eqref{eq:PU0}, one may notice that $\I_\e(U_\epsilon)$ is a perturbation term which is expected to converge uniformly to $1$ as $\epsilon\to 0$.
%
%We can now focus on finding $U_\e$. 
Considering the terms $I_\e$ as a perturbation, we divide problem \eqref{eq:main_fep} by $I_\e(U_\epsilon)(z)$ which is positive, and we take the logarithm on each side. Then we obtain the following equation, considering \eqref{eq:nec_lep} :
\begin{equation}\label{eq:def_Gammaep}
 U_\e(z)-2U_\e(\bz)+U_\e(0)  = \log \left( \dfrac{\iep{0}+ \m(z)}{\iep{z}} \right)
\end{equation}
It would be tempting to transform \eqref{eq:def_Gammaep} into a fixed point problem by inverting the linear operator in the left-hand-side. However, the latter is not invertible as it contains linear functions in its kernel. Therefore we are led to consider linear contributions separately.

Our main strategy is to decompose the unknown $U_\e$ under the form 
\begin{align}\label{eq:necshape Ueps}
\boxed{ U_\e (z) = \gamma_\e z + V_\e(z),}
\end{align}
with $V_\e \in \Eaz$. This is consistent with the analytic shape of our statement in \eqref{eq:def_U0}, where $\gamma_0$ and $V_0$ have quite different features with respect to the function $m$.

Next, it is natural to differentiate \eqref{eq:def_Gammaep}.
One ends up with the following recursive equation for every $z \in \R$ 
\begin{equation}\label{eq:rec_Gammaep}
\partial_z  U_\e (z) - \p_z U_\e(\bz)= \p_z \left[ \log \left( \dfrac{I_\e(U_\e)(0) +  m}{I_\e(U_\e)(z)} \right) \right] (z) .
\end{equation}
One simply deduces that, if $U_\e$ exists and is regular, then we must have: 
\begin{equation}
\ds \p_z U_\e(z)= \p_z U_\e (0) + \sum_{k \geq 0}\p_z \left[ \log \left( \dfrac{I_\e(U_\e)(0) +  m}{I_\e(U_\e)(z)} \right) \right]  (2^{-k}z).
\end{equation}
One can formally integrate back the previous equation to obtain
\begin{equation}\label{eq:nec_ueph}
\ds  U_\e(z)= U_\e(0) + \p_z U_\e (0)z + \sum_{k\geq 0} 2^k \log \left( \dfrac{\iep{0}+ m}{I_\e(U_\e)(z)} \right) ( 2^{-k}z) .
\end{equation}
At this stage we formally identify :
\begin{itemize}[label=$\triangleright$]
\item $U_\e(0)=0$, since $U_\e \in \Ea$. This is not a loss of generality by homogeneity since $F_\e$ is itself defined up to a multiplicative constant in problem \eqref{eq:main_fep}. 
\item $\gamma_\e = \p_z U_\e(0)$. In fact this is part of the decomposition \eqref{eq:necshape Ueps} since $V_\e \in \Eaz$.
\end{itemize}
The real number $\gamma_\e$ is unknown at this stage, but it needs to verify some compatibility condition  to make the series converging in \eqref{eq:rec_Gammaep}--\eqref{eq:nec_ueph}. In particular, if we evaluate \eqref{eq:rec_Gammaep} at $z=0$ we obtain that $\gamma_\epsilon$ must satisfy
\begin{equation}\label{eq:nec impgammaep}
0 = \p_z I_\e(\gamma_\e \cdot + V_\e) (0).
\end{equation}
We will solve \eqref{eq:nec impgammaep} using an implicit function theorem in order to recover the value $\gamma$ associated with a given $V$. 
Beforehand, we introduce the following notation:
\begin{definition}[Finite differences operator $\D$]$ $\\
We define the finite differences functional $\D$ as
\begin{equation*}
\D(V)(y_1,y_2,z) = V(\bz) - \frac12 V(\bz +\epsilon y_1) - \frac12 V(\bz +\epsilon y_2)\, , \quad \bz = \dfrac{z}{2}\, .
\end{equation*}
\end{definition}

We introduce the following auxiliary functional which makes the link between $\gamma_\e$ and $V$.
\begin{definition}[Auxiliary function $\Je$]$ $\\
We define the functional $\Je: \R \times \Eaz  \rightarrow  \R$ as follows 
\begin{equation}\label{eq:def_Jeps}
\Je(g,V)= \dfrac{1}{\e^2 \sqrt{2}\pi }\iint_{\R^2} \exp\left[-Q(y_1,y_2) - \e g (y_1+y_2) +2 \D(V)(y_1,y_2,0)  \right] \D(\p_z V)(y_1,y_2,0) \,   d y_1 d y_2. 
\end{equation}
\end{definition}
The implicit relationship \eqref{eq:nec impgammaep} is equivalent to $\Je(\gamma_\e,V_e) = 0$. From this perspective, 
the following result is an important preliminary step.
\begin{prop}[Existence and uniqueness of $\gamma_\e$]\label{existunique gammaeps}$ $\\
For any ball $K \subset \Eaz$, there exists $\e_K$, such that for all $\e \leq \e_K$ and for any $V  \in K $, there exists a unique solution $\gamma_\e(V) $ to the equation :
\begin{align*}
\text{Find $\gamma\in (-\RK,\RK)$ such that:} \quad\Je(\gamma,V) = 0,
\end{align*}
where the bound $\abs{\gamma_\e(V)} < \RK$ is defined as 
\begin{align}\label{def G+G-}
\RK=  \max \left( \ds  \dfrac{  \ds \nalph{K}   \ \iint _{\R^2} \exp(-Q (y_1,y_2) ) \left ( y_1^2  +   y_2^2  \right ) dy_1dy_2 +8}{ 2 \p_z^2 m(0)}   ; \nalph{K} \right)\, .
\end{align}
\end{prop}

Next we define the main quantity we will work with: the double integral $I_\e$ which is the rescaled infinitesimal operator. For convenience we define it as a mapping on $\Eaz$. It is compatible with \eqref{eq:def_Ieps_int} because of the decomposition  \eqref{eq:necshape Ueps}. 
\begin{definition}[Auxiliary functional $\I_\e$]\label{def Ieps}$ $\\
We define the functional $\I_\e: \Eaz   \rightarrow   \mathcal C^3(\R)$ as follows
\begin{align}\label{eq:def_Ieps}
 \I_\e(V)(z) =  \dfrac{\ds  \iint_{\R^2} \exp\big(-Q(y_1,y_2)-\e \gamma_\e(V)(y_1+y_2)  +2\D(V)(y_1,y_2,z) \big)d y_1 d y_2  }
                                 {\ds\sqrt{\pi }  \int_\R  \exp\left( - y^2/2 -\e \gamma_\e (V) y + V(0)- V(\epsilon y) \right)  d y} . 
\end{align}
\end{definition}
Finally, in view of \eqref{eq:nec_ueph} and \eqref{eq:necshape Ueps}, we see that $V_\e$ must be a solution of this implicit equation :
\begin{equation}\label{eq:nec_Veps}
V_\e (z) =\sum_{k\geq 0} 2^k \log \left( \dfrac{\I_\e(V_\e)(0)+ m}{\I_\e(V_\e)(z)} \right) ( 2^{-k}z) , \ \hbox{ for every } z \in \R. 
\end{equation}
This justifies the introduction of our central mapping, upon which our fixed point argument will be based. 
\begin{definition}[Fixed point mapping]\label{def:H}$ $\\
We define the mapping $\H: \Eaz \rightarrow \Eaz$ as follows
\begin{equation}\label{eq:defH}
\H(V)(h) =  \sum_{k \geq 0} 2^k \log \left( \dfrac{\I_\e(V) (0)+ m(2^{-k} h)}{\I_\e(V) (2^{-k}h)} \right) .
\end{equation} 
\end{definition}

\subsection{Reformulation of the problem}
We are now in position to write our main result for this Section:

\begin{thm}[Existence and uniqueness of the fixed point]\label{existunique fixedpoint}$ $\\
There is a ball $K_0 \subset \Eaz$ and a positive constant $\e_0$ such that for every $\e\leq \e_0$, the mapping $\H$ admits a unique fixed point in $K_0$.
\end{thm}

To conclude, it is sufficient to check that solving problem \eqref{eq:PUeps}, on the ball $K_0$, and seeking a fixed point for $\H$ in $K_0$ are equivalent problems for $\e\leq \e_0$ small enough.

%\vc{ICI J'AI CHANGE L'ORDRE DES ENONCES 2.7 ET 2.8 POUR MIEUX COLLER AU RAISONNEMENT DE LA SECTION 6. (EN PARTICULIER LE CHOIX DE $K_0'$ EST IMPORTANT.}

\begin{prop}[Reformulation of the problem \eqref{eq:PUeps}]\label{solutionisfixedpoint}$ $\\
There is a ball $K_0'$ of $\Ea$,  and a positive constant $\e_0'$ such that for every $\e\leq \e_0'$, the following statements are equivalent:
\begin{itemize}[label=$\triangleright$]
\item $(\lambda_\e,U_\e)$ is a solution of the problem \eqref{eq:PUeps} in $\R \times K_0'$. %and $\p_z^2 U_\e(0) \geq \dfrac{\p_z^2 m(0)}{2}$.
\medskip
\item $U_\e= \gamma_\e(V_\e) \cdot + V_\e$, with $V_\e \in \Eaz\cap K_0'$, $\H(V_\e) = V_\e$, and $\lambda_\e = \I_\e(V_\e)(0)$.
\end{itemize}
Moreover, the statement of \cref{existunique fixedpoint} holds true in the set $\Eaz\cap K_0'$.
\end{prop}

The main mathematical difficulties are stacked into \cref{existunique fixedpoint}.
The rest of the article is organized as follows :
\begin{itemize}[label=$\triangleright$]
\item In \cref{sec existunique gammaeps}, we  justify why the function $\gamma_\e$ is well defined in Proposition \ref{existunique gammaeps}.\medskip
\item Then in \cref{sec prop Ieps}, we provide the main properties and the key estimates of the nonlocal operator $\I_\e$. We point out why this term plays the role of a perturbation between problem \eqref{eq:PUeps} and problem \eqref{eq:PU0}. In \cref{sec a priori estimates} we prove crucial contraction estimates.\medskip
\item Those estimates are the main ingredients of the proof of properties of $\H$ in \cref{sec properties H}:  most notably the finiteness of $\H(V)$, and the fact that $\H$ is a contraction mapping.\medskip
\item  This allows us to establish the proof of \cref{existunique fixedpoint,solutionisfixedpoint}, and finally to come back to the proof of our main result \Cref{convergence PUeps} in the \cref{sec proof reformulation,sec proof convergence PUeps}.\medskip
\item  Section  \ref{sec large dim} is devoted to those specific arguments that require an extension to the higher dimensional case $d>1$.
\end{itemize} 
 \section{Well-posedness of the implicit function $\gamma_\e$}\label{sec existunique gammaeps}
 \subsection{Heuristics on finding $\gamma_\e$}\label{sec heuristicsgammaeps}
We consider $V \in \Eaz$, and we look for solutions $\gamma_\e$ of $\Je(\gamma_\e,V) = 0$, or equivalently :
\begin{equation}\label{eq:heurgam_equ}
0 = \dfrac{1}{\e^2\sqrt{2}\pi} \ds  \iint_{\R^2} \exp \left[-Q(y_1,y_2)  - \e \gamma_\e (y_1+y_2) +2 \D(V)(y_1,y_2,0) \right] \Big( \D(\p_z V)(y_1,y_2,0)\Big)  d y_1 d y_2  ,
\end{equation}
in accordance with \eqref{eq:def_Jeps}.  We will see here  how  a Taylor expansion of the right-hand-side around $\e = 0$ helps to  understand why it defines a unique $\gamma_\e$  in a given interval for small $\e$. We will show formally why $\Je(\cdot,V)$ can be uniformly approximated by a non-degenerate linear function for small $\e$.

We expand the right-hand-side with respect to $\epsilon$:
\begin{align*}
\nonumber \ds & \dfrac{1}{\e^2\sqrt{2}\pi}\iint_{\R^2} \exp \left[-Q(y_1,y_2) \right] \exp \left[ - \e \gamma_\e (y_1+y_2) + 2 \D(V)(y_1,y_2,0)   \right] \Big( \D(\p_z V)(y_1,y_2,0) \Big)  d y_1 d y_2 \\
& 
= 
- \dfrac{1}{\e^2 \sqrt{2}\pi}\iint_{\R^2} \exp \left[-Q(y_1,y_2) \right]   \left[1 - \e \gamma_\e (y_1+y_2)  + o\left (\e \right ) \right] \nonumber 
\\ 
&\qquad\qquad\qquad \qquad\qquad\qquad  
\qquad\;\; \times \left ( \frac\e2 \left ( y_1  + y_2 \right ) \p_z^2 V(0) + \dfrac{\e^2}4 \left (y_1^2 + y_2^2\right ) \p_z^3 V(0) + o(\e^2) \right )  d y_1 d y_2 
\\ 
& =\frac1{\e^2} \left(  \dfrac{\e^2}{2}  \gamma_\e \p_z^2 V(0) - \e^2   \dfrac{3 \p_z^3 V(0)}{8}  +  o(\e^2) \right) .
\end{align*}
Then solving 
\begin{align*}
0 =   -    \dfrac{3 \p_z^3 V(0)}{8}+ \dfrac{1}{2}  \gamma_\e \p_z^2 V(0) + o(1)   ,
\end{align*}
we get the expression :
\begin{equation}\label{eq:gammaepsd=1}
\gamma_\e \underset{\e \to 0}{\sim} \dfrac{3}{4}\dfrac{\p_z^3V(0)}{\p_z^2 V(0)}.
\end{equation}
These heuristics are consistent with the statement in Theorem \ref{convergence PUeps}, up to the relation between $V_0$ and $m$ that can be easily read out from \eqref{eq:gamma0value}. Note that the denominator involves $\p_z^2 V(0)$, so that the local convexity of $V$ should be controlled uniformly during our construction. This is the purpose of the restriction in $\Eaz$ \eqref{eq:E0}. In the following, we provide estimates that turn these heuristics into a rigorous proof.

\subsection{Proof of \cref{existunique gammaeps}}\label{sec:proof of prop24}
The aim of this section is to prove the existence and uniqueness of  $\gamma_\epsilon(V)$ stated in  \cref{existunique gammaeps}. We first start with a Lemma providing some useful estimates on the function $\Je$. Combining these estimates with a continuity and monotonicity arguments, we will be able to prove the \cref{existunique gammaeps}.

\begin{lem}[Estimates of $\Je$]\label{estimate Jeps}$ $\\
For any ball $K \subset \Eaz$, there exists $\e_K>0$, such that for all $\e \leq \e_K$ and  $V \in K$, the following estimate holds true for all $g$ in the interval $(-\RK, \RK)$:
\begin{align}
\label{eq:estim_Jeps}  
\Je(0,V) & = - \dfrac{1}{4 \sqrt{2}\pi} \iint _{\R^2} \exp(-Q (y_1,y_2) ) \left[ y_1^2 \p_z^3 V(\e \tilde{y_1}) +   y_2^2 \p_z^3 V(\e \tilde{y_2}))  \right] dy_1dy_2  +    O(\e) , \\
\label{eq:estim_Jeps'}  \p_g \Je(g,V)& =\dfrac{\p^2_z V(0)}{2} +  O(\e),
\end{align}
where, in the former expansion, the variable $\tilde{y_i}$ is a by-product of Taylor expansions and is such that $\abs{\tilde{y_i}} \leq \abs{ y_i}+1$.
\end{lem}

\begin{remark}$ $\\
We prove the uniqueness of $\gamma_\e$ on a  uniformly bounded interval. One may think it is a strong restriction not to look at large $\gamma_\e$. It is in fact a natural restriction as we have by definition   $\gamma_\e = \p_zU_\e(0)$, and $\p_z U_\e \in L^\infty$ in our perturbative setting. 
\end{remark}
We postpone the proof of the technical \cref{estimate Jeps} at the end of this section and we first use it to prove the \cref{existunique gammaeps}: \begin{proof}[Proof of \cref{existunique gammaeps}] 
Let $K$ be a ball of $\Eaz$ and $V \in K$. We deduce from \cref{estimate Jeps} that $|\Je(0,V)| \leq \GK + 1$, where 
\begin{equation*}
\GK = \dfrac{\nalph{K}}{4 \sqrt{2}\pi } \ \iint _{\R^2} \exp(-Q (y_1,y_2) ) \left ( y_1^2  +   y_2^2  \right ) dy_1dy_2\,,
\end{equation*}
for $\e$ small enough. Integrating \eqref{eq:estim_Jeps'} with respect to $g$, we obtain 
\begin{equation*}
\Je(g,V) = \Je(0,V) + \dfrac{\p^2_z V(0)}{2} g + O(\e)\, , 
\end{equation*}
where it is important to notice that the perturbation $ O(\e)$ is uniform with respect to $\e$ for $g\in (-\RK,\RK)$ and $V\in K$.   
Since $V\in \Eaz$, we know that $\partial_z^2V(0)\geq \partial_z^2m(0)>0$. Therefore, $\Je$ is uniformly increasing with respect to $g$ on $(-\RK,\RK)$. Moreover, the choice of $\RK$ is such that 
\begin{equation*}
\Je(\RK,V) \geq - 1 - \GK + \dfrac{\p^2_z m(0)}{2} \RK +  O(\e) = 1 +  O(\e) >0\,,   
\end{equation*}
for $\e$ small enough, and similarly, $\Je(-\RK,V) < 0$. Finally, there exists a unique  $\gamma_\e(V)$ satisfying $\Je(\gamma_\e(V),V)= 0$ because $\Je$ is continuous with repect to $g$ for $V\in \Eaz$. 
\end{proof}

\begin{proof}[Proof of \cref{estimate Jeps}]
Let $K$ be a ball of $\Eaz$ of radius $\nalph{K}$. In \cref{sec heuristicsgammaeps}, we have used formal Taylor expansions to get a formula for $\gamma_\e(V)$, morally valid when $\e = 0$. The idea here is to write exact rests to broaden the formula for small but positive $\e$.\medskip

\noindent$\triangleright$ \textbf{Proof of expansion \eqref{eq:estim_Jeps}.}
Let us pick $V\in K$ and $\epsilon>0$. Recall the expression of $\Je(0,V)$ :
\begin{equation*}%\label{eq:defJeps}
\Je(0,V) = \dfrac{1}{\e^2 \sqrt{2}\pi } \ds  \iint_{\R^2} \exp \left[-Q(y_1,y_2)  + 2 \D(V)(y_1,y_2,0)     \right] \Big( \D(\p_z V)(y_1,y_2,0) \Big)  d y_1 d y_2 .
\end{equation*}
We perform two Taylor expansions, namely: 
\begin{equation}\label{eq:Taylor1}
\begin{cases}
2 \D(V)(y_1,y_2,0)   = - \dfrac{\e^2}{2}\left( y_1^2 \p_z^2 V(\e \tilde{y_1}) + y_2^2 \p_z^2 V(\e \tilde{y_2}) \right)\medskip\\
\ds \D(\p_z V)(y_1,y_2,0)  = - \dfrac{\e (y_1+y_2)}{2} \p_z^2 V(0)- \dfrac{\e^2}{4} (y_1^2 \p_z^3 V(\e \tilde{y_1}) +   y_2^2 \p_z^3 V(\e \tilde{y_2})),
\end{cases}
\end{equation}
where $\tilde{y_i}$ denote some generic number such that $\abs{\tilde{y_i}} \leq \abs{y_i}$ for $i=1,2$. 
Moreover, we can write
\begin{equation}\label{eq:P}
\exp(-\e^2 P) = 1 -\e^2 P \exp(-\theta\e^2 P)\, , \quad P = \dfrac{1}{2}\left( y_1^2 \p_z^2 V(\e \tilde{y_1}) + y_2^2 \p_z^2 V(\e \tilde{y_2}) \right)\,,\quad |P|\leq \frac12 \left (y_1^2 + y_2^2\right ) \|V\|_\alpha\, ,
\end{equation}
for some $\theta = \theta(y_1,y_2)\in (0,1)$. Combining the expansions, we find:
\begin{multline*}
\Je(0,V) = \dfrac{1}{\e^2 \sqrt{2}\pi} \ds  \iint_{\R^2} \exp \left[-Q(y_1,y_2)\right ] \left ( 1  -\e^2 P \exp(-\theta\e^2 P)  \right ) \\  \times\left ( - \dfrac{\e (y_1+y_2)}{2} \p_z^2 V(0)- \dfrac{\e^2}{4} (y_1^2 \p_z^3 V(\e \tilde{y_1}) +   y_2^2 \p_z^3 V(\e \tilde{y_2}))  \right )  d y_1 d y_2 \, .
\end{multline*}
The crucial point is the cancellation of the $O(\e^{-1})$ contribution due to the symmetry of $Q$:
\begin{equation}\label{eq:cancellation}
   \iint_{\R^2} \exp (-Q(y_1,y_2))  (y_1 + y_2) d y_1 d y_2 = 0\, .
\end{equation}
So, it remains 
\begin{multline*}
\Je(0,V) = - \dfrac{1}{4 \sqrt{2}\pi} \iint _{\R^2} \exp(-Q (y_1,y_2) ) \left[ y_1^2 \p_z^3 V(\e \tilde{y_1}) +   y_2^2 \p_z^3 V(\e \tilde{y_2}))  \right] dy_1dy_2 \\ + \frac\e{2 \sqrt{2}\pi} \iint _{\R^2} \exp(-Q (y_1,y_2) ) P \exp(-\theta \e^2 P)  (y_1+y_2)  \p_z^2 V(0)    dy_1dy_2 \\
+ \frac{\e^2}{4\sqrt{2}\pi}  \iint _{\R^2} \exp(-Q (y_1,y_2) ) P \exp(-\theta \e^2 P) \left ( y_1^2 \p_z^3 V(\e \tilde{y_1}) +   y_2^2 \p_z^3 V(\e \tilde{y_2}) \right )         dy_1dy_2
\end{multline*}
Clearly the last two contributions are uniform $O(\e)$ for $V\in K$ and $\e \leq \e_K$ small enough. Indeed, the term $P$ is at most quadratic with respect to $y_i$ \eqref{eq:P}, so $Q + \theta\e^2 P$ is uniformly bounded below by a positive quadratic form for $\e$ small enough.    
\medskip

\noindent$\triangleright$ \textbf{Proof of expansion \eqref{eq:estim_Jeps'}.} 
The first step is to compute the derivative of $J$ with respect to $g$:
\begin{multline*}
\p_g \Je(g,V) = -\frac{1}{\e \sqrt{2}\pi} \ \iint_{\R^2} \exp\left[-Q(y_1,y_2) - \e g (y_1+y_2) + 2 \D(V)(y_1,y_2,0)    \right]\\ 
\times (y_1+y_2) \left[\D(\p_z V)(y_1,y_2,0) \right]  d y_1 d y_2.
\end{multline*}
Similar Taylor expansions as above yields:
\begin{multline*}
\p_g \Je(g,V) = -\frac{1}{\e \sqrt{2}\pi}  \iint_{\R^2} \exp\left[-Q(y_1,y_2)\right ] \left ( 1 - \e P' \exp(-\theta\e P') \right )  \\ \times 
(y_1+y_2) \left(- \dfrac{\e (y_1+y_2)}{2} \p_z^2 V(0)- \dfrac{\e^2}{4} (y_1^2 \p_z^3 V(\e \tilde{y_1}) +   y_2^2 \p_z^3 V(\e \tilde{y_2})) \right)  d y_1 d y_2,
\end{multline*}
where $P' = g(y_1 + y_2) +   y_1 \p_z  V(\e \tilde{y_1}) + y_2  \p_z  V(\e \tilde{y_2})$. Interestingly, the leading order term does not cancel anymore, and it remains:
\begin{equation*}
\p_g \Je(g,V) = \frac{1}{2 \sqrt{2}\pi} \left (   \iint_{\R^2}  \exp\left[-Q(y_1,y_2)\right ](y_1+y_2)^2 \, dy_1dy_2\right )   \p_z^2 V(0) +  O(\e)\, .
\end{equation*}
The justification that the remainder is a uniform $ O(\e)$ is similar as above, except that now $P'$ has a linear part depending on $g$, but the latter is assumed to be bounded a priory by $\RK$.
\end{proof} 

\section{Analysis of the perturbative term $\I_\e$}\label{sec prop Ieps}
\subsection{Lispchitz continuity of some auxiliary functionals}\label{sec reg prop Ieps}\label{sec regprop Ieps}
The function $\I_\e$ is crucially involved in the definition of the mapping $\H$. Thus to prove any contraction property on this mapping we will need Lipschitz estimates about  $\I_\e$ and the three first derivatives of its logarithm. But first we show that  $\I_\e$ really plays the role of a perturbative term between    problem \eqref{eq:PUeps} and problem \eqref{eq:PU0}  that converges to $1$ uniformly as $\epsilon\to0$.

\begin{prop}[Estimation of $\I_\e$]\label{estim Ieps}$ $
\\
For every $K$ ball of $\Eaz$, for every $\delta>0$, there exists a constant $\e_\delta$ that depends only on $K$ and $\delta$, such that for every $\e \leq \e_\delta$  and for every $V \in K$ :
\begin{align*}%\label{eq:deltabounds_Ieps}
(\forall z\in \R)\quad  1-\delta \leq \I_\e(V)(z) \leq 1+\delta\,.
\end{align*}
\end{prop}
\begin{proof}
Let  $V$ in $K$. For $\e \leq \e_K$, one can apply \cref{existunique gammaeps} which gives $\abs{\gamma_\e(V) } \leq \RK$. Next it is enough to write that :
\begin{align*}
& \underline{C}_\e := \dfrac{\ds  \iint_{\R^2} \exp\left[-Q(y_1,y_2) - 2 \e R_K (\abs{y_1}+\abs{y_2})  \right]d y_1 d y_2  }{\ds\sqrt{2 } \pi \int_\R   \exp\left( - y^2/2 - 2 \e R_K \abs{y}   \right) d y}  \leq  \I_\e(V)(z), \text{ and }  \\
& \I_\e(V)(z) \leq \dfrac{\ds  \iint_{\R^2} \exp\big(-Q(y_1,y_2) + 2 \e R_K(\abs{y_1}+\abs{y_2}) 	  \big)d y_1 d y_2  }{\ds\sqrt{2 } \pi \int_\R   \exp\left( - y^2/2 - 2 \e R_K \abs{y}  \right) d y} := \overline{C}_\e.
 \end{align*}
We deduce from this lower and upper estimates that the whole $\I_\e(V)$  converges uniformly to $1$ as $\e \to 0$.  
\end{proof}

Next, we show  Lipschitz continuity of various quantities of interest. 

\begin{prop}[Lipschitz continuity of $\gamma_\e$]\label{contraction gamma}$ $
\\
For every ball $K\subset \Eaz$, there exist constants $\Cgamma$, and $\e_K $, depending only on $K$, such that for all $\e \leq \e_K$, $V_1, V_2 \in K$
%, if we denote $V_1-V_2 := \triangle V$ :
\begin{equation*}%\label{eq:prop_cont_gamma}
\abs{\gamma_\e(V_1)-\gamma_\e(V_2)} \leq \Cgamma \nalph{V_1-V_2}.
\end{equation*}
\end{prop}
\begin{proof} 
Let $K$ be a ball of $\Eaz$, and let $V_1,V_2 \in K$. Let denote $\GammaI = \gamma_\e(V_i)$ for $i = 1,2$.
We argue by means of Fréchet derivatives: let $s\in (0,1)$, $\gamma_s = s \gamma_1 + (1-s)\gamma_2$, $V_s =  s V_1 + (1-s)V_2$, and consider the following computation:
\begin{equation}\label{eq:dsJ}
\dfrac{d}{ds} \Je(\gamma_s , V_s) = \p_\gamma \Je(\gamma_s,V_s) (\gamma_1 - \gamma_2) + D_V \Je (\gamma_s,V_s)\cdot  (V_1 - V_2)\,,
\end{equation}
where the Fréchet derivative of $\Je$ with respect to $V$ is:
\begin{multline*}
D_V \Je(\gamma,V)\cdot H  =   \dfrac{1}{\e^2\sqrt{2}\pi}    \iint_{\R^2} \exp \left[-Q(y_1,y_2)- \e \gamma (y_1+y_2) + 2 \D(V)(y_1,y_2,0)     \right] \\ 
\qquad\qquad\qquad\qquad\qquad\qquad\qquad \times \Big (   2 \D(H)(y_1,y_2,0)  \Big ) \Big( \D(\p_z V)(y_1,y_2,0) \Big)  d y_1 d y_2  \\ + 
 \dfrac{1}{\e^2 \sqrt{2}\pi}    \iint_{\R^2} \exp \left[-Q(y_1,y_2)- \e \gamma (y_1+y_2) + 2 \D(V)(y_1,y_2,0)     \right]   \Big( \D(\p_z H)(y_1,y_2,0) \Big)  d y_1 d y_2
\end{multline*}
We perform similar Taylor expansions as in \eqref{eq:Taylor1},
\begin{equation*}
 2 \D(W)(y_1,y_2,0)  = 
\begin{cases}
-\e  (y_1 + y_2) O\left (\|\p_z W\|_\infty \right )\\
-\e (y_1 + y_2)\p_z W(0) - (\e^2/2)\left ( y_1^2 + y_2^2 \right ) O\left (\|\p_z^2 W\|_\infty \right )
\end{cases}
\end{equation*} 
either for $W =V, H\in \Eaz$, or $W = \p_z V, \p_z H$. We deduce that 
\begin{multline}\label{eq:DVJ}
D_V \Je(\gamma,V)\cdot H  = \dfrac{1}{\e^2 \sqrt{2}\pi}    \iint_{\R^2} \exp \left[-Q(y_1,y_2)- \e \gamma (y_1+y_2) -\e  (y_1 + y_2) O\left (\|\p_z V\|_\infty \right )   \right] \\ \times 
\bigg[ \left (  - \frac{\e^2}2\left ( y_1^2 + y_2^2 \right ) O\left (\|\p_z^2 H\|_\infty \right )\right ) \Big( -\e  (y_1 + y_2) O\left (\|\p_z^2 V\|_\infty \right ) \Big)    \\ + 
   \Big(  -\e  (y_1 + y_2)\p_z^2 H(0) - \frac{\e^2}2\left ( y_1^2 + y_2^2 \right ) O\left (\|\p_z^3H\|_\infty \right )    \Big) \bigg]  d y_1 d y_2
\end{multline}
We proceed as in the previous section for the exponential term: there exists $\theta = \theta(y_1,y_2)\in (0,1)$ such that 
\begin{equation*}
\exp (- \e P' ) = 1 - \e P' \exp(-\theta \e P' )\, , \quad \text{where}\quad P' = \gamma (y_1+y_2) +  (y_1 + y_2) O\left (\|\p_z V\|_\infty \right )     
\end{equation*}
Again, the crucial point is  the cancellation of the $O(\e^{-1})$ contribution in \eqref{eq:DVJ}, as in \eqref{eq:cancellation}.
What remains is of  order one or below, and one can easily show that there exists $C_K$ such that 
\begin{align*}
\abs{D_V \Je(\gamma,V)\cdot H}&  \leq C_K \left (  \|\p_z^3H\|_\infty + |\p_z^2 H(0)| \left ( |\gamma + \|\p_z V\|_\infty | \right ) \right.\\ 
& \quad + \left. \e \|\p_z^2H\|_\infty \|\p_z^2V\|_\infty + \e \|\p_z^3H\|_\infty \left ( |\gamma + \|\p_z V\|_\infty | \right )   \right )\\
& \leq C_K \|H\|_\alpha \, , 
\end{align*}
provided $\e \leq \e_K$ is small enough.

On the other hand, we have already established that $\p_\gamma \Je(\gamma,V) = \p_z^2 V(0)/2 +  O(\e)$ in Lemma \ref{estimate Jeps}. Consequently, integrating \eqref{eq:dsJ} from $s = 0$ to $1$, we find:
\begin{equation*}
0 = \Je(\gamma_1,V_1) - \Je(\gamma_2,V_2) = \left ( \frac{\p_z^2 V(0)}2 +  O(\e) \right ) (\gamma_1 - \gamma_2) + \left ( \int_0^1 D_V \Je (\gamma_s,V_s)\cdot  (V_1 - V_2)\, ds\right )\, .
\end{equation*}
We deduce from the previous estimates and the local convexity condition in \eqref{eq:E0} that 
\begin{equation*}
|\gamma_1 - \gamma_2| \leq C_K \left ( \dfrac{2 }{\p_z^2m(0)} + C_K\e \right ) \|V_1 - V_2\|_\alpha \, ,  
\end{equation*}
for some $C_K$ and $\e \leq \e_K$ small enough.
\end{proof} 

In turn, \cref{contraction gamma} implies the Lipschitz continuity of $\I_\e$ as a function of $V$.

\begin{prop}[Lipschitz continuity of $\I_\e$]\label{cont Ieps}$ $\\
For every ball $K$ of $\Eaz$, there exist constants $\e_K$, $C_K$ depending only on $K$,  such that for all $\e \leq \e_K $, $V_1, V_2 \in K$, 
\begin{equation}\label{eq:cont_Ieps}
\sup_{z\in \R}\abs{\I_\e(V_1)(z)-\I_\e(V_2)(z)} \leq   \e C_K \nalph{ V_1 - V_2}.
\end{equation} 
\end{prop}
\begin{proof}
The Lipschitz continuity of $\I_\e$ with respect to $V$ can be proven by composition of Lipschitz functions. With the same notations as in the proof of \cref{contraction gamma}, and with the shortcut notation $\I_\e = A_\e/B_\e$ to separate the numerator from the denominator in \eqref{eq:def_Ieps} we have, 
\begin{multline*}
\dfrac{d}{ds} A_\e(V_s)(z) = - \iint_{\R^2}  \QV(y_1,y_2, z) \left ( \e \dfrac{d}{ds} \gamma_\e (V_s) (y_1 + y_2) \right.\\ 
\left. + \dint_{\bz}^{\bz+\e y_1} \p_z (V_1-V_2)(z') dz' -\dint_{\bz}^{\bz+\e y_2}  \p_z (V_1-V_2)(z') dz'  \right ) \, dy_1 dy_2\,,
\end{multline*}
where we have simply written $V(\bz+\e y_1) - V(\bz)  = \int_{\bz}^{\bz+\e y_1} \p_z V(z')\, dz'$, and where $\QV$ denotes the exponential weight:
\begin{equation*}
\QV(y_1,y_2, z) =  \dfrac1{\sqrt{2}\pi} \ds \exp\left[-Q(y_1,y_2)  -\e\gamma_\e(V)(y_1+y_2) + 2 \D(V)(y_1,y_2,z)  \right] \, .
\end{equation*}
We deduce that $A_\e$ is such that:
\begin{equation*}
(\forall z)\quad \left | \dfrac{d}{ds} A_\e(V_s)(z)\right | \leq \e  \iint_{\R^2}  \QV(y_1,y_2, z) \left (    \Cgamma \| V_1 - V_2\|_\alpha +    \| V_1 - V_2\|_\alpha \right ) \left ( |y_1| + |y_2|\right )  \, dy_1 dy_2\,.
\end{equation*}
As the weight $\QV$ is uniformly close to a positive quadratic form for small $\e$, we find that the numerator has a Lipschitz constant of order $\e$ uniformly with respect to $z$: 
\begin{equation*}
\sup_{z\in \R}\abs{A_\e(V_1)(z)-A_\e(V_2)(z)} \leq \e C_K \nalph{ V_1 - V_2}\, .
\end{equation*}
The same holds true for the denominator $B_\e$. In addition, a direct by-product of the proof of \cref{estim Ieps} is that $A_\e$ and $B_\e$ are uniformly bounded above and below by positive constants for $\e$ small enough. Consequently, the quotient $\I_\e = A_\e/B_\e$ is Lipschitz continuous. 
\end{proof}

It is useful to introduce the probability measure $d\QV$ induced by the exponential weight $\QV$:
\begin{multline*}%\label{eq:def_dQU}
 d\QV(y_1,y_2,z) = \dfrac{\QV(y_1,y_2,z)}{\ds \iint_{\R^2} \QV(\cdot, \cdot, z)} \\
 =  \dfrac{\ds \exp\left[-Q(y_1,y_2)  -\e\gamma_\e(V)(y_1+y_2) + 2 \D(V)(y_1,y_2,z)  \right] }                               {\ds\iint_{\R^2} \exp\left[-Q(y_1,y_2)-\e\gamma_\e(V)(y_1+y_2)  + 2 \D(V)(y_1,y_2,z)    \right]  d y_1 d y_2 }. 
\end{multline*}
As a consequence of the previous estimates, we obtain the following one:
\begin{lem}[Lipschitz continuity of $d\QV$]\label{contraction measure lem}$ $\\
For every ball $K$ of $\Eaz$, there exist constants $\e_K$, $C_K$ depending only on $K$,  such that for all $\e \leq \e_K $, $V_1, V_2 \in K$, 
\begin{multline}\label{eq:cont_measure}
\sup_{z\in \R}\abs{\dQVA(y_1,y_2,z)-\dQVB(y_1,y_2,z)}  \\
\leq   \e C_K \norm{V_1-V_2}_\alpha  ( 1+ \abs{y_1} + \abs{y_2})    \exp \Big(-Q(y_1,y_2) +    2 \e R_K (\abs{y_1}+\abs{y_2})    \Big).
\end{multline} 
Furthermore, under the same conditions, we have the following bound, uniform with respect to $z \in \R $: 
\begin{equation}\label{eq:control dQV}
d\QV(y_1,y_2,z) \leq \dfrac{1}{4}\exp\left[-Q(y_1,y_2) +  2 \e R_K (\abs{y_1}+\abs{y_2})   \right].
\end{equation}
\end{lem} 
%We now proceed to show the \cref{contraction measure lem} which was the key to most estimates we did.
\begin{proof} 
We first prove \eqref{eq:control dQV}: the function $\QV$  is such that 
\begin{equation*}%\label{eq:control dQV}
\QV(y_1,y_2,z) \gtrless \dfrac1{\sqrt{2}\pi} \exp\left[-Q(y_1,y_2) \mp    2 \e R_K (\abs{y_1}+\abs{y_2})  \right]. 
\end{equation*}
Therefore, its integral over $(y_1,y_2)\in \R^2$ converges to 1 as $\e\to 0$, and there exists $\e_K$ depending on $K$ such that $\iint \QV(y_1,y_2,z)\,dy_1 dy_2 \geq 4/{\sqrt{2}\pi}$ for $\e \leq \e_K$. This leads to \eqref{eq:control dQV}. 

In order to obtain \eqref{eq:cont_measure}, we proceed as in the proof of  \cref{cont Ieps}, as the denominator of $d\QV$ is the numerator $A_\e$ of $\I_\e$. For the Lipschitz continuity of the numerator of $d\QV$, we find:
\begin{equation*}
(\forall z)\quad \left | \dfrac{d}{ds} G_\e^{V_s}(y_1,y_2,z)\right | \leq \e   G_\e^{V_s} (y_1,y_2, z) \left (    \Cgamma \| V_1 - V_2\|_\alpha +    \| V_1 - V_2\|_\alpha \right ) \left ( |y_1| + |y_2|\right )   \,.
\end{equation*}
We deduce that the quotient $d\QV =\QV/{A_\e(V)}$ is also Lipschitz continuous: 
\begin{align*}
\abs{\dQVA - \dQVB} &\leq \abs{\dfrac{G_\e^{V_1} - G_\e^{V_2}}{A_\e(V_1)} + \dfrac{A_\e(V_2) - A_\e(V_1)}{A_\e(V_1)A_\e(V_2)} G_\e^{V_2}}\\
& \leq \e C_K\| V_1 - V_2\|_\alpha \left ( |y_1| + |y_2|\right ) \exp \Big(-Q(y_1,y_2) +    2 \e R_K(\abs{y_1}+\abs{y_2})  \Big)\\
& \quad + \e C_K \| V_1 - V_2\|_\alpha  \exp \Big(-Q(y_1,y_2) +    2 \e R_K (\abs{y_1}+\abs{y_2})  \Big)
\, . 
\end{align*}
This concludes the proof of \eqref{eq:cont_measure}. 
\end{proof}
 
To conclude, we have established in this section that $\I_\e$ is a perturbative term, both in the uniform sense $\I_\e (V) \to 1$, and in the Lipschitz sense: $\mathrm{Lip}_V \I_\e =  O(\e)$. In addition, we have proven a similar Lipschitz smallness property for a probabilty distribution $d\QV$ that will appear frequently in our contraction estimates. 

\subsection{Contraction properties (first part)}\label{sec a priori estimates}
On the way to estimating the fixed point mapping $\H$ \eqref{eq:defH}, we need good estimates on the logarithmic derivatives of $\\I_\e$. For that purpose, we introduce the following quantities for $i =1,2,3$:
\begin{equation}\label{def:W}
W^{(i)}_\e (V) (z)= \dfrac{\p_z^i \I_\e(V)(z)}{\I_\e(V)(z)} \, .
\end{equation}

For the sake of conciseness, we omit sometimes the dependency with respect to $y_1,y_2$ in the notations, as for instance: $d\QV(y_1,y_2,z) = d\QV(z)$. 
The following notation with a  duality bracket is useful:
\begin{align*}
\left\langle  d\QV(z) , f  \right\rangle = \iint_{R^2} d\QV(y_1,y_2,z) f(y_1 , y_2)\, dy_1 dy_2 .
\end{align*}
Indeed, for any $V \in \Eaz $, we have:
\begin{equation}\label{eq:W1}
W_\e^{(1)}(V)(z)= \left\langle d\QV(z), \D(\p_z V)(z)
 \right\rangle.
\end{equation}
Similarly:
\begin{equation*}
W_\e^{(2)}(V)(z) =   
\left\langle d\QV(z), \dfrac{1}{2} \D(\p_z^2 V)(z) + \left(  \D(\p_z V)(z)
 \right)^2  
 \right\rangle.
\end{equation*}
And finally :
\begin{equation}\label{eq:W3V}
W_\e^{(3)}(V)(z) =    
\left\langle d\QV(z),  \dfrac{1}{4} \D(\p_z^3 V)(z)  + \left( \D(\p_z V)(z)\right)^3  
+  3 \D(\p_z V)(z) \left( \dfrac{1}{2}\D(\p_z^2 V)(z)\right)    \right\rangle.
\end{equation}
In order to obtain estimates on $W^{(i)}$ it seems natural from the previous pattern of  differentiation to  begin with estimates on the symmetric difference of the derivatives of $V$.
\begin{lem}\label{estim derivatives}
For any $V\in \Ea$, and $(y_1,y_2) \in \R^2$, we have:
\begin{align}
\label{estim der Ea 1} \ds & \sup_z \ (1+|z|)^\alpha \abs{ \D(\p_z V)(y_1,y_2,z) }  \leq \e 2^\alpha  \Vert V  \Vert_\alpha  \left[\vert y_1 \vert+ \vert y_2 \vert + \e^{\alpha}\vert y_1 \vert^{1+\alpha} + \e^{\alpha}\vert y_2 \vert^{1+\alpha} \right],\\ 
\label{estim der Ea 2} \ds & \sup_z \ (1+|z|)^\alpha \abs{ \dfrac{1}{2} \D(\p_z^2 V)(y_1,y_2,z) }  \leq  \e 2^{\alpha-1}  \Vert V \Vert_\alpha  \left[\vert y_1 \vert+ \vert y_2 \vert + \e^{\alpha}\vert y_1 \vert^{1+\alpha} + \e^{\alpha}\vert y_2 \vert^{1+\alpha} \right],\\
\label{estim der Ea 3} \ds & \sup_z \ (1+|z|)^\alpha \abs{  \dfrac{1}{4} \D(\p_z^3 V)(y_1,y_2,z)  }  \leq 2^{\alpha -1}  \Vert V \Vert_\alpha   \left(1+  \frac{\e^\alpha}4  \left[\vert y_1 \vert^{\alpha}+ \vert y_2 \vert^{\alpha} \right] \right).
\end{align}
\end{lem}
It is important to notice that the first two right-hand-sides (resp. first and second derivatives) are of order $\e$. The third one is larger but controlled by $2^{\alpha-1}<1$. This is the first occurrence of the contraction property we are seeking. This is the main reason why we make the analysis up to the third derivatives.

\begin{proof}
We introduce the additional notation $\vphia(z) = (1+|z|)^\alpha$. First, since $\bz=z/2$, we have $\vphia(z) \leq 2^{\alpha} \vphia(\bz)$. 
%The rest of the arguments will be well-built Taylor expansions and control of the rests using the space $\Ea$. 
\medskip

\noindent$\triangleright$ \textbf{Proof of \eqref{estim der Ea 1}.}
By Taylor expansions, we have:
\begin{equation*}
\vphia(z) \left| \D(\p_z V)(y_1,y_2,z)\right| \leq   2^{\alpha} \vphia(\bz) \left| \dfrac{\e y _1}{2} \p^2_z V(\bz+\e \tilde y_1)+\dfrac{\e y_2 }{2} \p^2_z V(\bz+\e \tilde y_2)  \right|,
\end{equation*}
where $\abs{ \tilde{y_i} } \leq \abs{y_i}$.
Using the definition of $\|\cdot \|_\alpha$ \eqref{eq:normalpha}, we obtain
\begin{equation*}
\vphia(\bz)  \left|  \e y _1   \p^2_z V(\bz+\e \tilde y_1) \right| \leq  \dfrac{\e |y_1 | \vphia(\bz)} {\vphia(\bz+\e \tilde{y_1}) }\nalph{V}  \leq  \dfrac{\e |y_1 | (1+|\e \tilde y_1| + |\bz + \e \tilde{y_1} |)^\alpha}{(1+|\bz+\e \tilde y_1|)^\alpha}\nalph{V}
\end{equation*}
Since we chose $\alpha < 1$,  $|\cdot|^\alpha$ is sub-additive. Thus, we get
\begin{multline*}
\vphia(\bz) \left|  \e y _1  \p^2_z V(\bz+\e \tilde y_1) \right| \leq \e |y_1 | \left (1 +  \dfrac{ |\e \tilde y_1|^\alpha}{(1+|\bz+\e \tilde y_1|)^\alpha}\right )\nalph{V} \\
\leq \e |y_1 | (1 +  |\e y_1|^\alpha)\nalph{V} \leq \e (|y_1|+|y_1|^{1+\alpha} ). 
\end{multline*}
By symmetry of the role played by $y_1$ and $y_2$, we have proven \cref{estim der Ea 1}.\medskip

\noindent$\triangleright$ \textbf{Proof of \eqref{estim der Ea 2}.}
The second estimate is a consequence of the first one, applied to the derivative of $V$. Notice that it is allowed as $\Eaz$ enables control of derivatives up to the third order.\medskip

\noindent$\triangleright$ \textbf{Proof of \eqref{estim der Ea 3}.} 
We must be a little more careful in the estimations of the third estimate \eqref{estim der Ea 3}, because we cannot go up to the fourth derivative in the Taylor expansions. This is why we do not have an $O(\e)$ bound, but we gain a contraction factor instead. We have 
\begin{align*}
\vphia(z) \left|  \dfrac{1}{4} \D(\p_z^3 V)(y_1,y_2,z) \right|  & \leq 2^{\alpha}  \vphia(\bz) \left| \dfrac{1}{4} \p_z^3 V(\bz) -\dfrac{1}{8} \p^3_z V(\bz+\e y_1)-\dfrac{1}{8} \p^3_z V(\bz+\e y_2)\right| \\
& \leq  \dfrac{2^{\alpha}}{4} \nalph{V} +2^{\alpha}\vphia(\bz) \left|  \dfrac{1}{8} \p^3_z V(\bz+\e y_1)+\dfrac{1}{8} \p^3_z V(\bz+\e y_2)\right|
\end{align*}
We bound separately each term using again the sub-additivity of $\abs{\cdot}^\alpha$. For $\e \leq 1$  :
\begin{multline*}
\vphia(\bz) \left|  \dfrac{1}{8} \p^3_z V(\bz+\e y_1) \right | \leq   \dfrac{ \vphia(\bz)} { 8 \vphia(\bz+\e y_1) } \nalph{V} \\
\leq \dfrac{\nalph{V}}{8} \left( 1 + \dfrac{(|\e y_1 |)^\alpha}{(1+|\bz+\e y_1|)^\alpha} \right) \leq ( 1 + |\e y_1|^\alpha )\dfrac{\nalph{V}}{8}.
\end{multline*}
Summing it all up, one ends up with:
\begin{equation*}
\vphia(z) \left|  \dfrac{1}{4} \D(\p_z^3 V)(y_1,y_2,z) \right| \leq 2^{\alpha-1} \nalph{V} \left(1+ \frac14 \e^\alpha \left[\vert y_1 \vert^{\alpha}+ \vert y_2 \vert^{\alpha} \right] \right).
\end{equation*}
This is precisely \cref{estim der Ea 3}. 
\end{proof}

The following proposition is a first step towards contraction properties that will be established in \cref{sec properties H}. For convenience, we introduce the following notation:
\begin{equation}
\begin{cases}
\triangle W_\e^{(i)}  = W^{(i)}_\e(V_1)-W^{(i)}_\e(V_2) \\
\triangle V = V_1-V_2
\end{cases}
\end{equation}
\begin{prop}[Lipschitz continuity of $W_\e$ with respect to $V$]\label{unif boundsDeltaW}$ $\\
Let $K$ a ball of $\Eaz$, and $V_1, V_2 \in K $. There exists constants $\e_K$, $C_K$  depending only on $K$ such that for all $\e \leq \e_K$, we have:
\begin{align}
& \label{estim Delta 1} \ds \sup_z\; (1+|z|)^\alpha \vert\triangle W_\e^{(1)}	(z) \vert \leq \e C_K \Vert \triangle V \Vert_\alpha\\
& \label{estim Delta 2}  \ds \sup_z\; (1+|z|)^\alpha \vert\triangle W_\e^{(2)}(z) \vert \leq \e C_K \Vert \triangle V \Vert_\alpha, \\ 
& \label{estim Delta 3}    \ds \sup_z \;  (1+|z|)^\alpha \vert\triangle W_\e^{(3)}(z)\vert \leq  \left( 2^{\alpha-1} + \e^\alpha C_K \right)\norm{\triangle V}_\alpha.
\end{align}
%All the estimates can be rewritten with a $\mathcal O(\e)$ that depends only on $K$.
\end{prop}
It is also possible to get estimates on $W_\e^{(i)}(V)$ itself, with the same hypotheses. This is useful to prove the  invariance of certain subsets of $\Eaz$.
\begin{prop}\label{unif boundsW}$ $\\
With the same setting as in Proposition \ref{unif boundsDeltaW}, we also have:
\begin{align*}
&  \ds \sup_z \; (1+|z|)^\alpha \vert  W^{(1)}_\e(V)(z)\vert   \leq \e C_K \nalph{V},\\ 
& \ds \sup_z \; (1+|z|)^\alpha \vert  W^{(2)}_\e(V)(z)\vert  \leq \e C_K \nalph{V},\\
&  \ds \sup_z \;  (1+|z|)^\alpha \vert  W_\e^{(3)}(V)(z) \vert \leq \left( 2^{\alpha-1} + \e^\alpha C_K \right) \nalph{V}.
\end{align*} 
\end{prop}
%This proposition, although somewhat weaker is quite useful. First, we will use it to prove \cref{invariant subsets}, the same way \cref{unif boundsDeltaW} helps us showing that $\H$ is a contraction mapping  in \cref{contraction mapping}.  Secondly, it is precisely the estimate needed to prove the finiteness of $\H(V)$ \cref{finiteness Heps} in \cref{sec finiteness H}. 
We do not give the details of the proof of the latter Proposition, since it is a straightforward adaptation   of \cref{unif boundsDeltaW}. Actually, we cannot readily apply Proposition \ref{unif boundsDeltaW} to $(V_1,V_2) = (V,0)$ as $0\notin \Eaz$, because of the additional condition on $\p_z^2V(0)$ \eqref{eq:E0} which is required to prove boundedness and Lipschitz continuity of $\gamma_\e$.

\begin{proof}[Proof of \cref{unif boundsDeltaW}]
The proof of theses inequalities is quite tedious because of the numerous non-linear calculations. However, the technique is similar for each inequality, and consists in separating the fully non linear behavior from the quasi-linear parts of the left-hand-sides of \cref{estim Delta 1,estim Delta 2,estim Delta 3}. 
\medskip 

\noindent$\triangleright$ \textbf{Proof of \eqref{estim Delta 1}.}
This is the easiest part, because it is quasi-linear with respect to  $V$. Indeed, we have
\begin{equation*}
\triangle W_\e^{(1)}(z)  = \left\langle \dQVA(z), \D(\p_z V_1)(z)\right\rangle  
 - \left\langle \dQVB(z),  \D(\p_z V_2)(z) \right\rangle .
\end{equation*}
We reformulate it in two parts,  one involving  $V_1-V_2$, and the other involving $\dQVA-\dQVB$ :
\begin{equation}\label{eq:triangle W1}
\triangle W_\e^{(1)}(z)=  
  \left\langle \dQVB(z),\D(\p_z \triangle V)(z) \right\rangle + \left\langle \dQVA(z)-\dQVB(z), \D(\p_z V_1)(z) \right\rangle.
\end{equation}
For the first contribution in \eqref{eq:triangle W1}, we apply directly \cref{estim derivatives} to $V_1-V_2$:
\begin{align*}
\ds  (1+\abs{z})^\alpha \abs{ \left\langle \dQVB(z), \D(\p_z \triangle V)(z)  \right\rangle } &
 \leq \e 2^\alpha \norm{ \triangle V}_\alpha \left\langle \dQVB(z),  \left(\vert y_1 \vert+ \vert y_2 \vert + \e^\alpha \vert y_1 \vert^{1+\alpha} + \e^\alpha\vert y_2 \vert^{1+\alpha} \right) \right\rangle\\
 &\nonumber \leq \e C_K \norm{\triangle V}_\alpha.
\end{align*}
For the last inequality we used \cref{eq:control dQV}, which enables to bound uniformly the measure $d\QV$ with respect to~$z$.  
From \cref{contraction measure lem,estim derivatives}, there exists $\epsilon_K$ and $C_K$ such that for $\epsilon\leq \epsilon_K$, the second contribution in  the right-hand-side  \eqref{eq:triangle W1} satisfies
\begin{multline}\label{eq:truc}
\ds  (1+\abs{z})^\alpha   \abs{ \left\langle \dQVA(z)-\dQVB(z), \D(\p_z V_1)(z) \right\rangle } \\ 
\leq \e^2 C_K \norm{\triangle V}_\alpha  \Vert V_1 \Vert_\alpha   \Big\langle   (1+\abs{y_1} + \abs{y_2})    \exp (-Q(y_1,y_2) +  2\e \RK(\abs{y_1}+\abs{y_2})  )  , 
\\    
    \left(\vert y_1 \vert+ \vert y_2 \vert + \e^\alpha\vert y_1 \vert^{1+\alpha} + \e^\alpha\vert y_2 \vert^{1+\alpha} \right) \Big\rangle.
\end{multline}
The last integral is uniformly bounded for $\e$ small enough, involving moments of a Gaussian distribution. Therefore, the whole quantity is bounded by $\e^2 C_K \norm{\triangle V}_\alpha$, uniformly with respect to $z$.  This concludes the proof of \cref{estim Delta 1}.\medskip

\noindent$\triangleright$ \textbf{Proof of \eqref{estim Delta 2}.}
To begin with, we have
\begin{multline*}
\triangle W_\e^{(2)}(z) =    
\left\langle \dQVA(z), \dfrac{1}{2} \D(\p_z^2 V_1)(z) +  \left( \D(\p_z V_1)(z)\right)^2   \right\rangle \\
-\left\langle \dQVB(z), \dfrac{1}{2} \D(\p_z^2 V_2)(z)  + \left( \D(\p_z V_2)(z)\right)^2 \right\rangle.
\end{multline*}
We   split the difference into two, as in the previous part,
\begin{align*}
\triangle W_\e^{(2)}(z)
& = 
\left\langle \dQVB(z), \dfrac{1}{2} \D(\p_z^2 V_1)(z)  + \left( \D(\p_z V_1)(z)\right)^2  
- \dfrac{1}{2} \D(\p_z^2 V_2)(z) -   \left(\D(\p_z V_2)z)\right)^2  \right\rangle
\\
&\quad + \left\langle \dQVA(z)-\dQVB(z) ,  \dfrac{1}{2} \D(\p_z^2 V_1)(z)  + \left( \D(\p_z V_1)(z)\right)^2  \right\rangle  \\
 &= A + B 
\end{align*}
The first contribution can be rearranged as follows, by factorizing the difference of squares:
\begin{equation*}
A
= \left\langle \dQVB(z),\dfrac{1}{2} \D(\p_z^2 \triangle V)(z) 
+  \D(\p_z \triangle V)(z)  \D(\p_z (V_1+V_2))(z)   
\right\rangle .
\end{equation*}
The term involving $V_1+V_2$ is bounded uniformly in a crude way: $\|\D(\p_z (V_1+V_2))\|_\infty \leq 2 \|V_1+V_2\|_\alpha$ (in fact it is bounded by a $ O(\e)$ uniformly with respect to $z$, but this detail is omitted here). 
Then, we apply  \cref{estim derivatives} twice with $V_1-V_2$ to obtain:
\begin{align*}
\ds (1+\abs{z})^\alpha \abs{A} &  \leq   \e C_K \Vert  \triangle V  \Vert_\alpha  \left\langle \dQVB(z),   \left(\vert y_1 \vert+ \vert y_2 \vert +\e^\alpha \vert y_1 \vert^{1+\alpha} + \e^\alpha\vert y_2 \vert^{1+\alpha} \right) \right\rangle 
%\\
%&\quad  +   \left\langle \dQVB(z), (1+\abs{z})^\alpha \abs{\p_z \triangle V(\bz) -\dfrac{1}{2} \p_z \triangle V(\bz+\e y_1)-\dfrac{1}{2} \p_z \triangle V(\bz+\e y_2)}  \times \right. \\ 
%&\qquad\qquad\qquad  \left. \abs{\p_z (V_1+V_2)(\bz) -\dfrac{1}{2} \p_z (V_1+V_2)(\bz+\e y_1)-\dfrac{1}{2} \p_z (V_1+V_2)(\bz+\e y_2)} \right\rangle\\
%&  \leq  \e C_K \Vert \triangle V \Vert_\alpha  
\end{align*}

To estimate $B$, the term involving the difference of measures $d\QV$, we apply \eqref{eq:cont_measure} and \cref{estim derivatives}:
\begin{equation}\label{eq:estim PSI}
\ds  (1+\abs{z})^\alpha \abs{B}  \leq \left\langle \abs{\dQVA(z)-\dQVB(z) } , \e C \left ( \norm{V_1}_\alpha^2 + \nalph{V_2} \right )\left(\vert y_1 \vert+ \vert y_2 \vert + \e^\alpha\vert y_1 \vert^{1+\alpha} + \e^\alpha\vert y_2 \vert^{1+\alpha} \right)   \right\rangle.
\end{equation}
We find, exactly as above, that the   quantity $(1+\abs{z})^\alpha \abs{B} $ is bounded by $\e^2 C_K \norm{\triangle V}_\alpha$.
Combining both estimates on $A,B$, we deduce \cref{estim Delta 2}.
\medskip

\noindent$\triangleright$ \textbf{Proof of \eqref{estim Delta 3}.}
The full expression for $\triangle W_\e^{(3)}$ is as follows:
\begin{multline*}
\triangle W_\e^{(3)}(z) =   
\left\langle \dQVA(z),   \dfrac{1}{4} \D(\p_z^3 V_1)(z)  +  \left( \D(\p_z V_1)(z)\right)^3 + 3\D(\p_z V_1)(z)  \left( \dfrac{1}{2} \D(\p_z^2 V_1)(z)\right)\right\rangle\\
-\left\langle \dQVB(z), \dfrac{1}{4} \D(\p_z^3 V_2)(z) +  \left( \D(\p_z V_2)(z)\right)^3    +  3 \D(\p_z V_2)(z) \left( \dfrac{1}{2} \D(\p_z^2 V_2)(z)\right) \right\rangle.	  
\end{multline*}
We split again in two pieces, one involving $V_1-V_2$, and the other involving $\dQVA-\dQVB$:
\begin{equation*}
\triangle W_\e^{(3)}(V)(z) =  \left\langle \dQVB(z),  A_1 + A_2  + A_3 \right\rangle + \left\langle \dQVA(z)-\dQVB(z),B \right\rangle,
\end{equation*}
with 
\begin{align*}
A_1 &=  \dfrac{1}{4} \D(\p_z^3 \triangle V)(z)  \\
A_2 & = \left( \D(\p_z V_1)(z)\right)^3 -  \left( \D(\p_z V_2)(z)\right)^3  \\ 
& = 
\left( \D(\p_z \triangle V)(z)\right)
\bigg[ \left( \D(\p_z V_1)( z)\right)^2   
 +  \left( \D(\p_z V_2)(z)\right)^2   + \left( \D(\p_z V_1)(z)\right)\left( \D(\p_z V_2)(z)\right) \bigg]\\
A_3 & =  3 \D(\p_z^2 V_1)(z) \left( \dfrac{1}{2} \right)\D(\p_z V_1)(z)  -  3 \D(\p_z V_2)(z)\left( \dfrac{1}{2} \D(\p_z^2 V_2)(z)\right) 
\\
 & =   3 \D(\p_z V_1)(z) \left( \dfrac{1}{2} \D(\p_z^2 \triangle V)(z)\right)    + 3 \D(\p_z \triangle V)(z)\left( \dfrac{1}{2} \D(\p_z^2 V_2)(z)\right)\\
B &=   \dfrac{1}{4} \D(\p_z^3 V_1)(z)  +  \left(\D(\p_z V_1)(z)\right)^3  
  +  3\D(\p_z V_1)(z) \left( \dfrac{1}{2} \D(\p_z^2 V_1)(z)\right).
\end{align*}
We shall estimate all the contributions separately. Firstly, $A_1$ yields the contraction factor:
\begin{equation*}%\label{eq:estimD2}
(1+\abs{z})^\alpha \left\langle \dQVB(z),  \abs{ A_1} \right \rangle  \leq 2^{\alpha - 1}  \norm{\triangle V}_\alpha	\left\langle \dQVB(z),  1 + \frac{\e^\alpha}4  \left[\vert y_1 \vert^\alpha + \abs{y_2}^\alpha\right] \right\rangle \leq  \left ( 2^{\alpha - 1} + \e^\alpha C_K\right ) \norm{\triangle V}_\alpha\, .
\end{equation*}
The latter is the main contribution in \eqref{estim Delta 3}. The remaining terms are lower-order contributions with respect to $\e$.
For $A_2$, we have
\begin{align*}
(1+\abs{z})^\alpha \left\langle \dQVB(z),  \abs{ A_2} \right \rangle  & \leq \e 2^{\alpha} \norm{\triangle V }_\alpha  \left\langle \dQVB(z), \left( \nalph{V_1}^2+\nalph{V_2}^2+ \nalph{V_1} \nalph{V_2}  \right) \right.\\ 
& \qquad\qquad\qquad\qquad\qquad 
\times\left. \vphantom{\nalph{V_1}^2} \left[\vert y_1 \vert+ \vert y_2 \vert + \e^\alpha \vert y_1 \vert^{1+\alpha} + \e^\alpha\vert y_2 \vert^{1+\alpha} \right] \right \rangle\\
& \leq \e C_K \norm{\triangle V }_\alpha\, .
\end{align*}
For $A_3$, we have similarly  
\begin{equation*}
(1+\abs{z})^\alpha \left\langle \dQVB(z),  \abs{ A_3} \right \rangle    \leq \e C_K \norm{\triangle V }_\alpha\, .
\end{equation*}

It remains to control the term involving $B$. We argue as in \eqref{eq:truc} and \eqref{eq:estim PSI}:
\begin{multline*}
  \left\langle \abs{ \dQVA(z)- \dQVB(z) } ,  (1+\abs{z})^\alpha\abs{B} \right\rangle \\ \leq  \e C_K \norm{\triangle V}_\alpha \bigg\langle ( 1+ \abs{y_1} + \abs{y_2})    \exp (-Q(y_1,y_2) +   2\e \RK(\abs{y_1}+\abs{y_2})  )   ,\\
 2^{\alpha-1} \norm{V_1}_\alpha\left ( 1+  \frac{\e^\alpha}4  \left[\vert y_1 \vert^{\alpha}+ \vert y_2 \vert^{\alpha} \right] \right )  + C  \e  \left (  \norm{V_1}_\alpha^3  +  \norm{V_1}_\alpha^2\right ) \left(\vert y_1 \vert+ \vert y_2 \vert + \e^\alpha\vert y_1 \vert^{1+\alpha} + \e^\alpha\vert y_2 \vert^{1+\alpha} \right) \bigg\rangle.
\end{multline*}
The latter is controlled by $ \e C_K \norm{\triangle V }_\alpha$ for the same reasons as usual. 

Combining all the pieces together, we obtain finally \eqref{estim Delta 3}. 
\end{proof}

\section{Analysis of the fixed point mapping  $\H$}\label{sec properties H}
In this section we focus on the fixed point mapping $\H$ \eqref{eq:defH}, which is defined through an infinite series. 
%We recall that we investigate the existence of solutions to the equation 
%\begin{align*}
%\H(V):=\sum_{k \geq 0} 2^k \log \left( \dfrac{\I_\e(V) (0)+ m(2^{-k} h)}{\I_\e(V) (2^{-k}h)} \right)= V.
%\end{align*}
We are first concerned with the convergence of the series for $V\in \Eaz$. 
\subsection{Well-posedness of $\H$ on balls}\label{sec finiteness H}
Consider the following decomposition of each term of the series \eqref{eq:defH} in two parts, with the corresponding notations:
\begin{equation*}
\Gamma_\e(z) = \log \left( \dfrac{\I_\e(V)(0)+m(z)}{\I_\e(V)(0)} \right) - \log \left( \dfrac{\I_\e(V)(z)}{\I_\e(V)(0)}\right) = \Gammam(z)- \GammaI(z).
\end{equation*}
They have the following properties :
\begin{lem}\label{cond log Ieps+m Eaz}$ $ \\
For every ball $K\subset \Eaz$, there exists  $\e_K$  such that for any $\e \leq \e_K$, and $V\in K$, we have $\Gammam \in \Eaz$. Moreover, we have $(1+|z|)^\alpha\partial_z \Gammam \in L^\infty$.
\end{lem}
The proof of \cref{cond log Ieps+m Eaz} is a straightforward consequence of \cref{estim Ieps} and the assumptions on $m$ made in \cref{def m}, particularly \eqref{eq:log m}. 

\begin{lem}\label{cond log Ieps Eaz}$ $\\
For every ball $K\subset \Eaz$, there exists  $\e_K$  such that for any $\e \leq \e_K$, and $V\in K$, we have $\GammaI\in \Ea$, and $\p_z \GammaI (0) = 0$. Moreover, we have $(1+|z|)^\alpha\partial_z \GammaI \in L^\infty$.
\end{lem}
\begin{proof}
We begin by verifying the condition  $\p_z \GammaI (0) = 0$. This is in fact equivalent to the choice of  $\gamma_\e(V)$, as can be seen on the following computation:
\begin{align*}
\p_z \GammaI(0)=  \dfrac{\p_z \I_\e(V)(0)}{\I_\e(V)(0)} = W_\e^{(1)}(V)(0).
\end{align*}
Now, comparing \eqref{eq:heurgam_equ} with \eqref{eq:W1}, we see that $\p_z \GammaI(0) = 0$ is equivalent to $J(\gamma_\e(V),V) = 0$, provided $\e$ is small enough (for the quantities to be well defined). 

%For $\e$ uniformly small, the \cref{estim Ieps} tells us that $\I_\e(V)(0)$ is arbitrarily close to $1$ and as such does not vanish. Finally by construction of $\gamma_\e(V)$, we know that 
%\begin{align}
% \left[ \ds\sqrt{2 } \pi \int_\R N(y') \exp\left( V(0)- V(\epsilon y') -\e \gamma_\e y' \right) d y' \right] \p_z \I_\e(V)(0) = \Je(\gamma_\e(V),V) = 0.
%\end{align}
%In the end:$ \p_z \GammaI(0) = 0$.

Secondly, we need to get uniform bounds on the derivatives of $\GammaI$ to prove that it belongs to $\Ea$. The following formulas relate the successive logarithmic derivatives of $\I_\e(V)$ to the  $W_\e^{(i)}(V)$ introduced in \cref{def:W}:
\begin{align}
\p_z \GammaI(z) & =  W_\e^{(1)}(V)(z) \label{eq:pz gammaI}\\
\p_z^2 \GammaI(z) & =  W_\e^{(2)}(V)(z) - \left [W_\e^{(1)}(V)(z)\right ]^2, \label{eq:pz2 gammaI}\\
\p_z^3 \GammaI(z) & =  W_\e^{(3)}(V)(z) + 3W_\e^{(1)}(V)(z)W_\e^{(2)}(V)(z)  +2 \left [W_\e^{(1)}(V)(z)\right ]^3 . \label{eq:pz3 gammaI}
\end{align}
We can use directly the  weighted estimates in \cref{unif boundsW}, which include the algebraic decay of the first order derivative. Algebraic combinations are compatible with those estimates because $W_\e^{(i)}(V) \in L^\infty(\R)$. 
A fortiori those terms are all uniformly bounded and so we obtain that $\GammaI\in \Ea$.
%\begin{align}\label{eq:Gamma_Ea}
%\log \left( \dfrac{\I_\e(V)}{\I_\e(V)(0)}\right) \in \Ea.
%\end{align}
%Furthermore, we notice that we can easily show with the previous arguments that
%\begin{align}\label{eq:control nalph logI}
%\nalph{ \log \left( \dfrac{\I_\e(V)}{\I_\e(V)(0)}\right) } \leq \left( \frac{1}{2^{1-\alpha}} + \e^\alpha C_K \right) \nalph{V}.
%\end{align}
%Finally we have shown that we could apply \cref{existence sums} to $\GammaI $. 
\end{proof}
The main result of this section is the following one:
\begin{prop}[Convergence of the series $\H(V)$]\label{finiteness Heps}$ $\\
For every ball $K\subset \Eaz$, there exists  $\e_K$  such that for any $\e \leq \e_K$, and $V\in K$, the sum $\H(V)$ is finite.
\end{prop}
Before proving this statement, we first establish an  auxiliary technical lemma about  the following summation operator $S$:
\begin{equation*}
\S:  \Lambda \longmapsto \left( h \mapsto \ds \sum_{k \geq 0} 2^k \Lambda	 (2^{-k}h) \right). 
\end{equation*}

\begin{lem}[Existence of the sum]\label[lem]{existence sums}$ $\\
Take any function $\Lambda \in \Ea$ such that  $\p_z  \Lambda(0) = 0$. Then $\S(\Lambda)(h)$ is well-defined for  every $h\in \R$  . 
\end{lem}
\begin{proof}
%
%We split the sums in two parts as follows: \vc{let $N_h \in \N$ be the lowest integer such that $|h| \leq 2^{N_h}$.}
%We perform a Taylor expansion in the regime $k\geq N_h$. 
%We separate two regimes:$k\geq N_h$ and $k \leq N_h-1$. 
We perform a Taylor expansion: there exists $\tilde{h}_k$, such that 
$\Lambda (2^{-k}h) = \frac12 (2^{-k}h)^2  \p_z^2 \Lambda(2^{-k} \tilde{h_k})$. Therefore, we have immediately
\begin{equation*}
\ds \abs {\sum_{k \geq 0} 2^k \Lambda  (2^{-k}h) }   \leq \left (  h^2 \sum_{k \geq 0} 2^{-k} \right ) \left\| \p_z^2 \Lambda\right \|_{\infty} < \infty.
\end{equation*}
\end{proof}
One can now proceed to the proof of the finiteness of the sum of $\H$ in \cref{def:H}.
\begin{proof}[Proof of \cref{finiteness Heps}]
Let $K$ be the ball of $\Eaz$ of radius $\nalph{K}$ and take $V \in K$, $z \in \R$. To use the previous lemma, we first notice the  identity by definition: 
\begin{equation}\label{eq:relation HVS}
\H(V) = \S(\Gamma_\e).
\end{equation}
There are two conditions to verify in order to apply \cref{existence sums}:
\begin{align*}
\p_z  \Gamma_\e  (0)= 0, \text{ and } \Gamma_\e \in \Ea.
\end{align*}
Those properties are verified thanks to \cref{cond log Ieps+m Eaz,cond log Ieps Eaz}. The \cref{finiteness Heps} immediately follows.
\end{proof}

So far, we 
%Those proofs highlighted the role of our estimates on $W_\e^{(i)}(V)$ for the finiteness of $\H$, \vc{except that we 
have not used the algebraic decay condition which is part of the definition of $\Ea$. In the following lemma, we refine the estimate on $\S(\Lambda)\in \Ea$.  
This foreshadows the same result for the function $\H(V)$, as stated in the next section.
%Indeed we will use there, in the next section, segments of the proof of \cref{control sums}.
\begin{lem}[Better control of the series]\label{control sums}$ $
\\
Assume that $\Lambda\in \Ea$, that $\p_z \Lambda(0) = 0$, and that $(1+|z|)^\alpha\partial_z \Lambda\in L^\infty$. Then, $\S(\Lambda)$ belongs to $\Ea$, with a uniform estimate:
\begin{equation}\label{eq:control S alpha}
\nalph{  \S(\Lambda)} \leq C\max \left ( \|\Lambda\|_\alpha, \sup_{z\in \R}\ (1+|z|)^\alpha|\partial_z \Lambda(z)| \right )
\end{equation}
%Given any $ \Lambda \in \Ea$  such that in addition , the function $\S(\Lambda) $ lives in the space $\Ea$. In addition there exists an universal bound $C$ such that  :
%\begin{align*}
%\nalph{  \S(\Lambda)} \leq C \nalph{\Lambda}.
%\end{align*}
\end{lem}

There is some subtlety hidden here. In fact, we were not able to propagate the algebraic decay at first order from $\Lambda$ to $\S(\Lambda)$. What saves the day is that  we gain some algebraic decay of the first order derivatives somewhere in our procedure (see {\em e.g.} Proposition \ref{unif boundsW}).

\begin{proof}%[Proof of \cref{control sums}]
Recall the notation $\vphia (h)= (1 + \abs{h})^\alpha$. We begin with the uniform bound on the first derivative, which is the main reason why we have to impose algebraic decay in our functional spaces.\medskip

\noindent$\triangleright$ \textbf{Step 1: $ \p_z \S(\Lambda)$ is uniformly bounded.}
We split the sum  in two parts. Let $h \in \R$, and let $N_h \in \N$  be the lowest integer such that $|h| \leq 2^{N_h}$.
We consider the two regimes: $k> N_h$ and $k \leq  N_h$. In the former regime, a simple Taylor expansion yields 
\begin{equation}
\ds \abs{\sum_{k > N_h} \p_z \Lambda (2^{-k}h)} \leq  \ds \sum_{k > N_h} 2^{-k}| h |  \left \|\p_z^2 \Lambda \right \|_{\infty} \leq   \left \|\p_z^2 \Lambda \right \|_{\infty} ,
\end{equation} 
by definition of $N_h$.
In the regime  $k \leq N_h $,  we use the algebraic decay which is encoded in the space $\Ea$. If $|h|>1$, we have $N_h\geq 1$, and
\begin{multline*}
\ds \left | \sum_{k\leq N_h} \p_z  \Lambda (2^{-k}h)\right| \leq \ds \sum_{k\leq N_h } \frac{ \|\vphia \p_z \Lambda \|_{\infty}} {(1+2^{-k}|h|)^\alpha}\\ 
\leq \left (  \ds  \sum_{k\leq N_h } \dfrac{2^{k \alpha}}{|h|^\alpha} \right ) \|\vphia \p_z \Lambda \|_{\infty} = \left (  \dfrac{1}{|h|^\alpha} \dfrac{2^{(N_h+1) \alpha }-1}{2^\alpha-1} \right ) \|\vphia \p_z \Lambda \|_{\infty}.
\end{multline*}
By definition of $N_h$, we have $2^{N_h-1} < |h|$, so that the right-hand-side above is bounded by a constant that get arbitrarily large as $\alpha\to 0$ (hence, the restriction on $\alpha>0$):
\begin{equation}\label{eq:alpha collapses}
\ds \abs{ \sum_{k\leq N_h} \p_z \Lambda (2^{-k}h) } \leq \left ( \dfrac{4^\alpha}{2^\alpha - 1} \right ) \|\vphia \p_z \Lambda \|_{\infty}.
\end{equation}
The case $|h|\leq 1$ is trivial as the sum is reduced to a single term $\p_z  \Lambda (h)$ since $N_h = 0$.
\medskip

\noindent$\triangleright$ \textbf{Step 2: $ \vphia\abs{ \p_z^2 \S(\Lambda)}  $ is uniformly bounded.} This bound and the next one are easier.  
For any $h \in \R $, we have
%\begin{equation*}
%\ds \vphia(h) \sum_{k\geq 0} 2^{-k}  \abs{ \p_z^2 \Lambda (2^{-k}h) }.
%\end{equation*}
\begin{equation*}
\ds \vphia(h) \abs{ \sum_{k\geq 0} 2^{-k}  \p_z^2 \Lambda (2^{-k}h) } \leq  \left ( \ds \sum_{k \geq 0} 2^{-k}   \dfrac{\vphia(h)}{\vphia(2^{-k}h)}\right ) \|\vphia \p_z^2 \Lambda \|_{\infty} .
\end{equation*}
Since $1\geq 2^{-k}$, one obtains
\begin{equation*}
\ds \vphia(h) \abs{ \sum_{k\geq 0} 2^{-k}  \p_z^2 \Lambda (2^{-k}h) }  \leq \left(\sum_{k \geq 0}  2^{k (\alpha-1)} \right) \|\vphia \p_z^2 \Lambda \|_{\infty} = \left ( \dfrac{2}{2 - 2^\alpha} \right ) \|\vphia \p_z^2 \Lambda \|_{\infty}. 
\end{equation*}
The latter sum is finite since $\alpha<1$. 
\medskip

\noindent$\triangleright$ \textbf{Step 3: $ \vphia\abs{ \p_z^3 \S(\Lambda)}  $ is uniformly bounded.} The proof is similar to the previous argument. 
\end{proof}

\subsection{Contraction properties (second part)}\label{sec cotraction properties}
In this section we  prove that $\H$ stabilizes some subset of $\Ea_0$. We first show that $\H$ maps balls into balls with incremental radius that do not depends on the initial ball (\cref{invariant subsets}). This property immediately implies the existence of an invariant subset for $\H$ (\cref{cor invariant subset}). Finally, we prove that the mapping $\H$ is a  contraction mapping for $\epsilon$ small enough (\cref{contraction mapping}). 
To completely justify the definition of $\H$, it remains to show that $\H(V) \in \Eaz$. We begin with the lower bound on the second derivative, which is for free.
%, which is the point of \vc{\cref{lowerbound p2HV}}. 
%The first step is the following lemma, which justifies the inclusion of the condition on $\p_z^2 V(0)$ in the subspace $\Eaz$.
\begin{lem}[Lower bound on $\p_z^2 \H(V)(0)$]\label{lowerbound p2HV}$ $
\\
For every ball $K\subset \Eaz$, there exists  $\e_K$  such that for any $\e \leq \e_K$, and $V \in K$, we have:
\begin{align*}
\p_z \H(V)(0) = 0, \quad \p_z^2 \H(V)(0) \geq \p_z^2 m(0).
\end{align*}
\end{lem}

\begin{proof}%[Proof of \cref{lowerbound p2HV}]
The identity $\p_z \H(V)(0) = 0$, and more particularly $\p_z \GammaI(0) = 0$ is a consequence of the choice of $\gamma_e(V)$ in \cref{existunique gammaeps}. Indeed, we have, by \eqref{eq:relation HVS},
\begin{align*}
\p_z \H(V)(0)  = \sum_{k \geq 0} \p_z \Gamma_\e(V)(0) = 0.
\end{align*}
%We insist upon the fact that it is a consequence of the definition of $\gamma_\e(V)$ in \cref{existunique gammaeps}. 

For the second estimate, a simple computation yields,  using $m(0)= \p_z m(0) = 0$:
\begin{equation*}
\ds \p^2_z \H(V) (0) = \sum_{k \geq 0} 2^{-k} \left[\dfrac{\p^2_z m(0)}{\I_\e(V)(0)} -	W_\e^{(2)}(V)(0)-W_\e^{(1)}(V)(0)^2 \right]
\end{equation*}
But since $V \in \Eaz$, one can use again the uniform estimates of  \cref{unif boundsW} to write that for $\e \leq \e_K$, that depends only on the ball $K$ :
\begin{equation}\label{eq:calc_d2Veps_int}
\ds \p^2_z \H(V) =  2 \dfrac{\p^2_z m(0)}{\I_\e(V)(0)}  +  O(\e),
\end{equation}
where $ O(\e)$ that depends only on the ball $K$. Then, we use \cref{estim Ieps} with $\delta = 1/3$ to deduce that for $\e$ small enough, we have $\I_\e( V) \leq {4}/{3}$. Then \eqref{eq:calc_d2Veps_int} can be simplified into 
\begin{equation*}
\ds \p^2_z \H(V)(0) \geq   \dfrac{3 \p^2_z m(0)}{ 2 } + O(\e).
\end{equation*}
Recall that $\p_z^2m(0)>0$ by assumption. Therefore, for $\e$ small enough, we get as claimed
\begin{equation*}
\ds \p^2_z \H(V)(0) \geq    \p^2_z m(0).
\end{equation*}
\end{proof}
\begin{remark}Considering the proof, another way to interpret the result is that automatically for any function $V \in \Ea$ such that $\p_z V(0) = 0$, the function $\H$ prescribes a lower bound on $\p_z^2 V(0)$. Since we are seeking a  fixed point $\H(V)= V$, we may as well put this condition in the subspace $\Eaz$ without loss of generality. 
%One can notice that it is also true for the condition $\p_z V(0)=0$, which is instantly verified by $\H(V)$ thanks to the definition of $\gamma_\e(\cdot)$ in  \cref{existunique gammaeps}.
\end{remark}

Finally, we can establish a first useful estimate on $\nalph{\H(V)}$, showing more than just its finiteness:
\begin{prop}[Contraction in the large]\label[prop]{invariant subsets}$ $
\\
For every ball $K \in \Eaz$, there exists an explicit constant $\kappa(\alpha)<1$, as well as $C_m$, $C_K$ and $\e_K$ that depend only on $K$ such that, for all $\e \leq \e_K$, and for every $V \in K$, 
\begin{align}\label{eq:bound nalphHeps}
\nalph{\H(V)} \leq C_m + (\kappa(\alpha)+\e^\alpha C_K)\nalph{V}.
\end{align}
\end{prop}

\begin{proof}%[Proof of \cref{invariant subsets}]
Let $K$ be the ball of $\Eaz$, and take  $V \in K$. For clarity we write respectively $\I_\e(h) $ and $W_\e^{(i)}(h)$ instead of $\I_\e(V) (h)$ and $W_\e^{(i)}(V) (h)$. %Recall that $\I_\e$ is the double integral that plays the role of a perturbation between problem \eqref{eq:PUeps} and problem \eqref{eq:PU0}, and $W_\e^{(i)}$ is its i$-th$ logarithmic derivative, see \cref{def Ieps,def:W}.  We are interested in 
Combining various estimates derived in Section \ref{sec finiteness H}, and particularly Lemma \ref{control sums} together with \cref{cond log Ieps+m Eaz,cond log Ieps Eaz}, we find that $ \H(V) = \S(\Gamma_\e) = \S(\Gammam) - \S(\GammaI)$ belongs to $\Ea$. However, the associated estimate \eqref{eq:control S alpha} is not satisfactory, at least for the $\S(\GammaI)$ and we need to re-examine the dependency of the constants upon $\e$ and $\alpha$.

%\begin{align*}
%\ds  \H(V)(h) & = \sum_{k \geq 0} 2^k \Big( \log(\I_\e(0)+ m(2^{-k}h)) - \log(\I_\e(2^{-k}h)) \Big)
%\end{align*}
%As in the proof of \cref{finiteness Heps}, we cut this sum in two part distinguishing the affine parts :
%\begin{align*}
%\ds  \H(V)(h) & = \sum_{k \geq 0} 2^k \left[ \log \left( \dfrac{\I_\e(0)+ m(2^{-k}h)}{\I_\e(0)} \right) - \log \left( \dfrac{\I_\e(2^{-k}h) }{\I_\e(0)}\right) \right] \\& 
%= \S \left(  \log \left( \dfrac{\I_\e(0)+ m}{\I_\e(0)} \right) \right)  - \S \left( \log \left( \dfrac{\I_\e }{\I_\e(0)}\right) \right) .
%\end{align*}
%We can apply \cref{control sums} to the sum $\S \left(  \log \left( \dfrac{\I_\e(0)+ m}{\I_\e(0)} \right) \right) $ because of our \cref{cond log Ieps+m Eaz} that established :
%\begin{align*}
%\log \left( \dfrac{\I_\e(0)+m}{\I_\e(0)}\right) \in \Ea.
%\end{align*}
%For the other term, we can also apply the \cref{control sums} since we have established in \eqref{eq:Gamma_Ea} that :
%\begin{align*}
%\log \left( \dfrac{\I_\e}{\I_\e(0)}\right) \in \Ea.
%\end{align*}
%However, this only yields 
%\begin{align*}
%\nalph{ \S \left( \log  \dfrac{\I_\e }{\I_\e(0)} \right)} \leq C \nalph{ \log \left( \dfrac{\I_\e }{\I_\e(0)} \right)}
%\end{align*}
%We wish to be more precise upon this bound, which is why we detail the proof of those estimates. 
The first and second derivatives of $\GammaI$ involve $W_\e^{(1)}$ and $W_\e^{(2)}$ which are both of order $\e C_K \nalph{V}$  thanks to \cref{unif boundsW}. Back to the proof of \cref{control sums}, the quantities $\|\vphia\p_z \GammaI\|_{\infty} $ and $\|\vphia\p_z^2 \GammaI\|_{\infty} $ are in fact  of order  $\e \nalph{\Lambda}$, and so are $\| \p_z\S(\GammaI)\|_{\infty} $ and $\| \vphia\p_z^2\S(\GammaI)\|_{\infty} $.

This cannot be extended readily to the third derivative as we lose the order $\e$ at this stage. However, \cref{unif boundsW} provides an explicit constant that is going to be used. From \eqref{eq:pz3 gammaI}, we have:
\begin{equation*}
\ds  \p^3_z \S \left( \GammaI \right)(h) =   \ds \sum_{k \geq 0}  4^{-k} \left[W_\e^{(3)}(2^{-k}h)+3W_\e^{(2)}(2^{-k}h)W_\e^{(1)}(2^{-k}h)+2W_\e^{(1)}(2^{-k}h)^3 \right]. 
\end{equation*}
The contributions involving  $W^{(1)}_\e$ and $W^{(2)}_\e$ are of order $\e$, and can be handled exactly as above. However, the linear term involving  $W_\e^{(3)}$ requires a careful attention. We obtain from \cref{unif boundsW} that $\vphia W_\e^{(3)}$ is bounded uniformly by $\left( 2^{\alpha-1} + \e^\alpha C_K \right) \nalph{V}$. Therefore, 
\begin{multline}\label{eq:contraction kappa}
\vphia(h) \abs{  \sum_{k\geq 0}  4^{-k}W^{(3)}_\e(2^{-k}h) } 
\leq \left( 2^{\alpha-1} + \e^\alpha C_K \right)     \left (  \sum_{k \geq 0} 4^{-k} \dfrac{\vphia(h)}{\vphia(2^{-k}h)} \right ) \nalph{V} \\
\leq 
\left( 2^{\alpha-1} + \e^\alpha C_K \right)     \left (  \sum_{k \geq 0}   2^{k(\alpha - 2)}   \right ) \nalph{V} = 
\left(  \dfrac{2^{\alpha+1}}{4 - 2^\alpha} + \e^\alpha C_K \right) \nalph{V} . 
\end{multline}
In view of the latter estimate, we define the explicit constant $\kappa(\alpha)$ as
\begin{equation}
\kappa(\alpha) = \dfrac{ 2^{1+\alpha}}{4-2^\alpha} .
\end{equation}
A simple calculation shows that $\kappa(\alpha)<1$ if and only if $\alpha < 2 - \log_2(3) \approx 0.415$. 
The choice of $\alpha \leq 2/5$ gives some room below this threshold. We conclude that $\nalph{\S(\GammaI)}\leq (\kappa(\alpha)+\e^\alpha C_K)\nalph{V}$.

The other contribution to $\H(V)$, namely $\S(\Gammam)$ can be bounded in an easier way. Indeed, we have  
\begin{equation}\label{eq:gamma m decomposed}
\Gammam  = \log(1 + m) + \log \left( 1 + \dfrac{m}{\I_\e(0)} \right) - \log(1 + m) = \log(1 + m) + \log \left( 1 + \dfrac{m}{1 + m}\left ( \dfrac1{\I_\e(0)} - 1\right )\right). 
\end{equation}
We define accordingly 
\begin{equation}
C_m = \max_{k = 1,2,3}  \left (  \left \|\vphia \dfrac{ \p_z^{k} m}{1 + m} \right \|_{\infty} \right )\, , 
\end{equation}
Moreover, \cref{estim Ieps} can be easily refined into $|\I_\e(0) - 1|\leq \e C_m \|V\|_\alpha$, using the definition of $R_K$ in \eqref{def G+G-}. Straightforward computations show that the last contribution in \eqref{eq:gamma m decomposed} can be estimated by $\e C_m \nalph{V}$. %[ICI LES ARGUMENTS SONT UN PEU RAPIDES CAR PROP 4.1 N'EST PAS REDIGEE COMME LES AUTRES ESTIMATIONS. ON PEUT LAISSER TEL QUEL, C'EST SOUS CONTROLE]

Combining the estimates obtained for  $\S(\Gammam)$ and $\S(\GammaI)$, we come to the conclusion:

\begin{align*}
\nalph{\H(V)} \leq C_m + (\kappa(\alpha) + \e^\alpha C_K )\nalph{V},
\end{align*}
\end{proof}

\cref{invariant subsets} calls an immediate corollary.
\begin{cor}[Invariant subset]\label{cor invariant subset}$ $
\\
There exist $K_0$ a  ball of $\Eaz$,  and $\e_0 $ a positive constant
such that for all $\e \leq \e_0$ the set $K_0$ is invariant by $\H$: 
\begin{align*}
\H(K_0) \subset K_0.
\end{align*}
\end{cor}
\begin{proof}%[Proof of \cref{cor invariant subset}]
Let $K_0$ be the ball of radius $R_0 = 2C_m/ (1 - \kappa(\alpha))$. We deduce from \cref{invariant subsets} that, for all $V\in K_0$,  
\begin{align*}
\nalph{\H(V)} \leq C_m + (\kappa(\alpha) + \e^\alpha C_{K_0} )R_0 &= C_m \left ( 1 + \dfrac{2\kappa(\alpha)}{1 - \kappa(\alpha)} \right ) + \e^{\alpha} C_{K_0} R_0 \\
&= C_m \left ( \dfrac{2}{1 - \kappa(\alpha)} - 1\right ) + \e^{\alpha} C_{K_0} R_0 \\
& = R_0 + C_m \left (  -1 + \dfrac{2\e^{\alpha} C_{K_0}}{1 - \kappa(\alpha)} \right ).
\end{align*}
Therefore, the choice
$\e_0  = \left( \dfrac{1-\kappa(\alpha)}{2 C_{K_0}} \right)^\frac1\alpha $ guarantees that $K_0$ is left invariant by $\H$.
\end{proof}

We are now in position to state the more important result of this section:
\begin{thm}[Contraction mapping]\label{contraction mapping}$ $
There exists a constant $C_{K_0}$ such that for any $\e\leq \e_0$, and every function $V_1$, $V_2 \in K_0$, the following estimate holds true
\begin{equation}\label{eq: contraction th}
\nalph{ \H(V_1)- \H(V_2)} \leq (\kappa(\alpha) + \e^\alpha C_K)\nalph{V_1-V_2}.
\end{equation}
\end{thm}
\begin{proof}
We denote by  $\triangle V$ the difference $V_1 - V_2$, again. 
The proof is analogous to \cref{invariant subsets}. 
For clarity we write respectively $\I_\e^i(h) $ instead of $\I_\e( V_i ) (h)$ and  $\triangle W_\e^{(i)}(h)$ instead of  $W_\e^{(i)}(V_1) (h)-W_\e^{(i)}(V_2) (h)$.  
%because there cannot be any confusion in this part. 
We decompose $ \triangle \H(V)$ as above: 
\begin{equation}
\triangle \H = \triangle \left (  \S(\Gammam) - \S(\GammaI)\right ) = \triangle \H^m  - \triangle \H^I. 
\end{equation}
%\begin{align*}
%\ds  \triangle \H (h) & = \sum_{k \geq 0} 2^k \left[ \left( \dfrac{\log(\I_\e^1(0)+ m(2^{-k}h))}{\I_\e^1(0)} -\dfrac{ \log(\I_\e^2(0)+ m(2^{-k}h))}{\I_\e^2(0)} \right) + \left( \dfrac{\log(\I_\e^2(2^{-k}h))}{\I_\e^2(0)}  - \dfrac{\log(\I_\e^1(2^{-k}h))}{\I_\e^1(0)} \right)  \right] \\
%& := \triangle \H^m(h) + \triangle \H^p(h).
%\end{align*}
We deal with $\triangle \H^m$ in the following lemma :
\begin{lem}\label{unif bound affpart diff}
There exists a constant $C_{0}$ such that for any $\e\leq \e_0$, and every function $V_1$, $V_2 \in K_0$, we have 
\begin{align*}
\nalph{\triangle \H^m} \leq \e C_0 \nalph{ \triangle V }.
\end{align*}
\end{lem}

\begin{proof}%[Proof of \cref{unif bound affpart diff}]$ $
Recall the following definition:
\begin{equation}
 \triangle \Gammam = \log\left ( \I_\e^1(0) + m \right ) - \log\left ( \I_\e^2(0) + m \right ) - \log\left (   \dfrac{\I_\e^1(0) }{\I_\e^2(0) } \right ) \, . 
\end{equation}
The first derivative has the following expression,
\begin{equation}\label{eq:triangle gammam}
\p_z  \triangle \Gammam  = -  \dfrac{\p_z m}{(\I_\e^1(0) + m)(\I_\e^2(0) + m)} \triangle \I_\e(0)\, .
\end{equation}
Clearly, $\I_\e^2(0) + m$ is bounded below, uniformly for $\e$ small enough. Therefore, we can repeat the arguments of \cref{control sums}, with $\Lambda = \log(\I_\e^1(0) + m)$ in order to get
\begin{equation}
\|\p_z  \S(\triangle \Gammam)\|_{\infty} \leq 
C_m  |\triangle \I_\e(0)|.
\end{equation}  
However, \cref{cont Ieps} yields that $|\triangle \I_\e(0)| \leq \e C_{0} \nalph{\triangle V}$.

The next order derivatives can be handled similarly. Indeed, the following quantities must be bounded uniformly by $\e C_{0} \nalph{\triangle V}$:
\begin{align*}
\ds  \vphia(h) \abs{ \sum_{k \geq 0} 2^{-k} \left[ \dfrac{\p_z^2 m(2^{-k}h)}{\I_\e^1(0)+m(2^{-k}h)}-\dfrac{\p_z^2 m(2^{-k}h)}{\I_\e^2(0)+m(2^{-k}h)} \right] }& \leq\e C_{0} \nalph{\triangle V}\\
\ds \vphia(h) \ds \abs{\sum_{k \geq 0} 2^{-k} \left[ \dfrac{\p_z m(2^{-k}h)^2}{(\I_\e^1(0)+m(2^{-k}h))^2}- \dfrac{\p_z m(2^{-k}h)^2}{(\I_\e^2(0)+m(2^{-k}h))^2} \right] }&  \leq \e C_{0} \nalph{\triangle V} \\
\ds  \vphia(h) \abs{ \sum_{k \geq 0} 4^{-k} \left[ \dfrac{\p_z^3 m(2^{-k}h)}{\I_\e^1(0)+m(2^{-k}h)}-\dfrac{\p_z^3 m(2^{-k}h)}{\I_\e^2(0)+m(2^{-k}h)} \right] }& \leq\e C_{0} \nalph{\triangle V} \\
\ds \vphia(h)   \abs{\sum_{k \geq 0} 4^{-k}\left [ \dfrac{ \p_z m(2^{-k}h)^3}{(\I_\e^1(0)+m(2^{-k}h))^3} -\dfrac{ \p_z m(2^{-k}h)^3}{(\I_\e^2(0)+m(2^{-k}h))^3}\right ] }&  \leq\e C_{0} \nalph{\triangle V},\\
\ds\vphia(h)  \abs{ \sum_{k \geq 0} 4^{-k}\left [ \dfrac{ \p^2_z m(2^{-k}h)\p_z m(2^{-k}h)}{(\I_\e^1	 (0)+m(2^{-k}h))^2} -\dfrac{  \p^2_z m(2^{-k}h)\p_z m(2^{-k}h)}{(\I_\e^2(0)+m(2^{-k}h))^2}\right ]} & \leq \e C_{0} \nalph{\triangle V}.
\end{align*}
The first and the third items are handled similarly as for the first derivative. The three other items are handled analogously. For the sake of concision, we focus on the second line:
We have,
\begin{multline*}
\dfrac{\p_z m(z)^2}{(\I_\e^1(0)+m(z))^2}- \dfrac{\p_z m(z)^2}{(\I_\e^2(0)+m(z))^2} \\
= \left[ \dfrac{\p_z m(z)}{\I_\e^1(0)+m(z)}+\dfrac{\p_z m(z)}{\I_\e^2(0)+m(z)} \right] \left[ \dfrac{- \p_z m(z) \triangle \I_\e (0) }{(\I_\e^1(0)+m(z))(\I_\e^2(0)+m(z))} \right]
\end{multline*} 
The first factor is uniformly bounded by assumption \eqref{eq:log m}, for $\e$ small enough. The second factor is the same as above, so we can conclude directly. 
\end{proof}

It remains to handle $\triangle \H^I$. We have the following formulas for the two first derivatives  \eqref{eq:pz gammaI}--\eqref{eq:pz3 gammaI}:
\begin{align*}
\ds  \p_z \triangle \H^I (h) & = \sum_{k \geq 0} \triangle W_\e^{(1)}(2^{-k}h), \\
\ds  \p^2_z \triangle \H^I(h) & = \ds   \sum_{k\geq 0} 2^{-k} \left[\triangle W_\e^{(2)}(2^{-k}h)-\triangle \left (  W_\e^{(1)}(2^{-k}h)^2\right )  \right]
\end{align*}
Finally  the formula for the third derivative is:
\begin{equation}\label{eq:delta 3 H}
\ds \p^3_z \triangle \H^I(h) =  \ds    \sum_{k \geq 0}  4^{-k} \left[\triangle W_\e^{(3)}(2^{-k}h)+3 \triangle \left (W_1^{(2)}W_1^{(1)}(2^{-k}h) \right )   +2 \triangle\left (  W_\e^{(1)}(2^{-k}h)^3\right )  \right].	
\end{equation}

%\vc{[ICI IL Y A QUELQUE CHOSE QUI M'ECHAPPE DANS LA VERSION PRECEDENTE. SAUF ERREUR, IL SUFFIT D'APPLIQUER LA PROP. 4.6 ET LEMME 5.5 DIRECTEMENT. POURQUOI REFAIRE L'ARGUMENT DU DECOUPAGE DES SOMMES ?]}

The combination of \cref{unif boundsDeltaW} and \cref{control sums} yields 
\begin{equation}
\|\p_z \triangle \H^I\|_{\infty} \leq \e C_0 \nalph{\triangle V}.  
\end{equation}
In the same way, we get the bound for the second derivative, using the factorization 
\begin{equation}
\triangle \left (  W_\e^{(1)}(z)^2\right ) =  \left ( W_\e^{(1)}(V_1)(z) + W_\e^{(1)}(V_2)(z) \right ) \triangle \left (  W_\e^{(1)}(z)\right )\, ,
\end{equation} 
together with the uniform bound in \cref{unif boundsW}. 

As in the proof of \cref{invariant subsets}, the third order derivative must be handled with care, as it does not yield a $ O(\e)$ bound.

Exactly as above, the contribution involving $\triangle W_\e^{(3)}$ in \eqref{eq:delta 3 H} is the one that yields the contraction factor, the remaining part being of order $O(\e) \nalph{\triangle V}$. Actually, we have precisely:

\begin{equation*}
\vphia(h) \ds \abs{ \sum_{k\geq 0}  4^{-k}  \triangle W_\e^{(3)}  (2^{-k}h)}  \leq  \left(  \kappa(\alpha) + \e^\alpha C_0 \right) \nalph{\triangle V} 
\end{equation*}
as in \eqref{eq:contraction kappa}. This concludes the proof of the main contraction estimate.
\end{proof}

\section{Existence of a (locally) unique $U_\e$, and convergence as $\e\to 0$.}

\subsection{Solving problem (\ref{eq:PUeps})  --  \cref{convergence PUeps}(i)}
\label{sec proof reformulation}

First of all, \cref{contraction mapping} immediately implies \cref{existunique fixedpoint}, that is the existence of a unique fixed point $\H(V_\e) = V_\e$ in the invariant subset $K_0$, for $\e \leq \e_0$. Note that $\e_0$ could possibly be reduced to meet the requirement of the last estimate in \eqref{eq: contraction th}.

However, due to the peculiar role played by the linear part $\gamma_\e(V_\e)$, it is convenient to enlarge slightly the set $K_0$. More precisely, after \cref{cor invariant subset} we define $K_0' $ the ball of radius 
\begin{equation}\label{eq:R0'}
R_0' = R_0 + \sup_{V\in K_0}|\gamma_\e(V)|\, .
\end{equation} 
It is clear that, up to reducing further $\e_0$ to $\e_0'$ in order to control the new constant $C_{K_0'}$, the set $K_0'$ is also invariant for $\e\leq \e_0'$. The same contraction estimate as in \cref{contraction mapping} holds, obviously. Furthermore, the fixed point on $K_0'$ coincides with the fixed point on the smaller ball $K_0$, by uniqueness.

Next, we show that finding this fixed point is equivalent to solving problem \eqref{eq:PUeps}, as claimed in \cref{solutionisfixedpoint}. We prove in fact the two sides of the equivalence. 
%We fix $K$ a ball of $\Ea$ and pick $\e >0$.

$\triangleright$ The easy part consists in saying that, being given $V_\e$ the unique fixed point in $K_0$, the function $U_\e = \gamma_\e(V_\e)\cdot + V_\e$ belongs to $K_0'$ by definition of $K_0'$ \eqref{eq:R0'}, and it solves problem \eqref{eq:PUeps} by construction.

$\triangleright$ On the other side, suppose that $(\lambda_\e,U_\e) \in \R \times K_0' $ is a solution of the problem \eqref{eq:PUeps}.
As in \cref{sec reformulation}, evaluating  \eqref{eq:PUeps} at $z=0$ yields the following necessary condition on $\lambda_\e$, since $m(0)=0$:
\begin{equation*}
\lep = \iep{0}.\footnote{ We use the notation $I_\e(U_\e)$ introduced in \cref{eq:def_Ieps_int}, that should not be confused with $\I_\e(V_\e)$. It is the purpose of the present argument to show that the two quantities do coincide.}
\end{equation*}
Then, we focus on $U_\e$. 
We decompose it as $U_\e = \gamma_U\cdot + V_U$, with $\gamma_U = \p_z U_\e(0) $, and $\p_z V_U(0)=0$.
Our purpose is threefold: $(i)$ first, we show that $\gamma_U = \gamma_\e(V_U)$, then $(ii)$ we prove that $V_U \in \Eaz$, and finally $(iii)$, we prove that $\H(V_U)= V_U$.

We can reformulate problem \eqref{eq:PUeps} as follows:
\begin{equation}\label{eq:PFV'}
 I_\e(\gamma_U \cdot + V_U)(0) + m(z) = I_\e(\gamma_U \cdot + V_U)\exp\left (V_U(z)-2V_U\left (\bz\right ) + V_U(0) \right ) .
\end{equation}
Since we assume $U_\e \in \Ea$, we can differentiate the previous equation, and evaluate it at $z=0$ to get :
\begin{align*}
\p_z I_\e(\gamma_U \cdot + V_U)(0)  = 0.
\end{align*}
As in Section \ref{sec reformulation}, a direct computation shows that $\gamma_U$ and $V_U$ are linked by the following relation:
%\begin{multline*}
%\partial_z I_\e(\gamma'_\e  \cdot + V_U)(0) = \\
%\dfrac{\ds  \iint_{\R^2} \! \! \exp\left[ -Q(y_1,y_2) +2 V_\e(0) - V_\e(\epsilon y_1) - V_\e(\epsilon y_2)-\e\p_z U_\e(0) (y_1+y_2) \right]  \Big(\p_z V_\e(0) - \dfrac{1}{2}\p_z V_\e(\e y_1)-\dfrac{1}{2}\p_z V_\e (\e y_2)\Big) d y_1 d y_2  }
%                                 {\ds\sqrt{2 } \pi \int_\R N(y') \exp\left( V_\e(0)-V_\e(\epsilon y')-\e \p_z U_\e(0)y' \right) d y'}
%\end{multline*} 
\begin{equation}\label{eq:J'}
0 = \Je(\gamma_U, V_U), 
\end{equation}
In order to invert this relationship, and deduce that $\gamma_U = \gamma_\e(V_U)$, it is important to prove that $V_\e \in \Eaz$, which amounts to showing that $\p_z^2 V_\e(0) \geq \p_z^2 m(0)$, the other conditions being clearly verified.

Differentiating the problem \eqref{eq:PUeps} twice, and evaluating at $z=0$, we get:
\begin{align*}
\p_z^2 m(0) & = \p_z^2 I_\e(U_\e)(0)  + I_\e(U_\e)(0) \dfrac{ \p_z^2 U_\e(0)}{2}. \nonumber \\
& = I_\e(U_\e)(0) \left( \dfrac{\p_z^2 I_\e(U_\e)(0)}{I_\e(U_\e)(0)} +  \dfrac{ \p_z^2 U_\e(0)}{2} \right).
\end{align*}
Then, using straightforward adaptations of \cref{estim Ieps,unif boundsW}, where $V$ should be replaced with $V_U\in \Ea$ and $\gamma_\e(V)$ should be replaced by $\gamma_U$, we find that  
\begin{align*}
\p_z^2 m(0)  \leq \frac32 \left( \e C_{K_0'} +  \dfrac{ \p_z^2 U_\e(0)}{2} \right).
\end{align*}
for $\e$ sufficiently small.
We deduce that the missing condition is in fact a consequence of the formulation \eqref{eq:PUeps}: 
\begin{align*}
\p_z^2 U_\e(0) \geq \p_z^2 m(0).
\end{align*}
By definition, $\p_z^2 U_\e(0) = \p_z^2 V_\e(0)$, so we have established that $V_\e \in \Eaz$.

Hence, we can legitimately invert \eqref{eq:J'}, so as to find $\gamma_U = \gamma_\e(V_U)$, where the function $\gamma_\e$ is defined in \cref{existunique gammaeps}. Since $U_\e\in K_0'$ by assumption, we have in particular $\nalph{V_U}\leq R_0'$. Of course, $V_U$ is the candidate of being the unique fixed point of $\H$ in $K_0'$ (but also in $K_0$). The proof of this claim follows the lines of \cref{eq:formal equivalence},   checking that all manipulations are justified.

%We can apply \cref{existunique gammaeps}, because $\p_z U_\e(0) \leq  \nalph{U_\e}\leq \RK$. This implies by \cref{existunique gammaeps} that necessarily $\p_z U_\e(0)= \gamma_\e(V_\e)$. Secondly, we prove that $V_\e \in \Eaz$. This is not straightforward with the definition of $V_\e$ since one needs to prove  $\p_z^2 V_\e(0) \geq \p_z^2 m(0)$. 

First, we  divide \eqref{eq:PFV'} by $I_U = I_\e(\gamma_U \cdot + V_U) = I_\e(\gamma_\e(V_U) \cdot + V_U) = \I_\e(V_U)$. According to \cref{estim Ieps}, this quantity is uniformly close to $1$ for $\e$  small, so it does not vanish. Taking the logarithm on both sides, we get for all $z \in \R$ :
\begin{equation*}%\label{eq:def_Gammaep_bis}
 V_U(z)-2V_U(\bz)+V_U(0)   = \log \left( \dfrac{I_U(0)+ \m(z)}{I_U(z)} \right) .
\end{equation*}
We differentiate the last equation to end up with the following recursive equation for every $z \in \R$ 
\begin{equation*}
\partial_z V_U (z) - \p_z V_U(\bz) = \p_z \log \left( \dfrac{I_U(0)+ \m}{I_U(z)} \right)(z) .
\end{equation*}
One simply deduces, that for all $z \in \R$, we necessarily have: 
\begin{equation*}
\ds \p_z V_U(z) = \p_z V_U (0)  + \sum_{k \geq 0} \log \left( \dfrac{I_U(0)+ \m(2^{-k}z)}{I_U(2^{-k}z)} \right) .
\end{equation*}
Note that the $\mathcal C^1$ continuity at $z=0$ is used here. Moreover, $\p_z V_U (0) = 0$ by definition of $V_U$.
The analysis performed in \cref{finiteness Heps} guarantees  that this sum is indeed finite. Finally, integrating back the previous identity yields 
\begin{equation*}
\ds  V_U(z)=   \sum_{k\geq 0} 2^k \log \left( \dfrac{I_U(0)+ \m(2^{-k}z)}{I_U(2^{-k}z)} \right) =  \sum_{k\geq 0} 2^k \log \left( \dfrac{\I_\e(V_U)(0)+ \m(2^{-k}z)}{\I_\e(V_U)(2^{-k}z)} \right)  .
\end{equation*}
The last expression is nothing but $\H(V_U)$, by definition \eqref{eq:defH}. Therefore, $V_U = \H(V_U)$ is the unique fixed point of $\H$ in $K_0'$. 

%\section{Proof of the main result \Cref{convergence PUeps}}\label{sec proof mainresult}
% \subsection{Existence and uniqueness of solutions}
%We first recall that we have proved that there exists a radius $R_\text{min}$ and $\e_{0}$ such that the mapping $\H$ admits a unique fixed point for every $\e \leq \e_{0}$ on each ball of radius greater that $R_{\text{min}}$. We now fix $\e \leq  \e_{0}$, and a ball $K_R$ of radius $R \geq R_\text{min}$. We call $V_\e \in \Eaz$ the fixed point given by \cref{existunique fixedpoint}. Then we define that $(\lambda_\e, U_\e)$ defined by
%\begin{align*}
%\lambda_\e & := \I_\e(V_\e)(0), \\
%U_\e (z) & := \gamma_\e(V_\e)z + V_\e(z) \text{ for all } z\in \R,
%\end{align*}
%with $\gamma_\e(V_\e)$ defined in \cref{existunique gammaeps}, $\I_\e$ defined by \cref{def Ieps}. Then $(\lambda_\e, U_\e)$ is a solution of problem \eqref{eq:PUeps} in $\R \times 2 K_R$. Suppose there exists an other solution of problem \cref{eq:PUeps} :  $(\lambda_\e' , U_\e') \in \R \times 2 K_R$. 
%
%Then, there exits $V_\e ' \in \Eaz $ a fixed point of $\H$ such that $U_\e' = \gamma_\e(V_\e ') + V_\e '$. By the triangle inequality $V_\e ' \in 4 K_R$. But then, by uniqueness in \cref{existunique fixedpoint}, $V_\e = V_\e '$, and so $U_\e = U_\e'$. Moreover, from \cref{solutionisfixedpoint}, it is straightforward that $\lambda_\e = \lambda_\e'$. Finally \cref{estim Ieps} allows us for $\e$ uniformly small to restrict $\lambda_\e$ to a finite interval $I_0$, which contains $1$. \qed

\subsection{Convergence of $(\lep, U_\e)$ towards $(\lambda_0,U_0)$ -- \Cref{convergence PUeps}(ii)}\label{sec proof convergence PUeps}
As previously, we decompose $U_\e = \gamma_\e\cdot + V_\e$, where $\gamma_\e$ stands for $\gamma_\e(V_\e)$. Firstly, we have $\lambda_\e = \I_\e(V_\e)(0)\to 1$, using \cref{estim Ieps}. Secondly, using an argument of diagonal extraction, there exists a subsequence $\e_n$, and a limit function $V_0$ such that 
\begin{align}
\label{eq:convegence pzVe}\lim_{\e\to 0} \p_z V_\e=  \p_z V_0 \,,\quad \text{in $L^\infty_\loc$},\\
\label{eq:convegence pz2Ve}\lim_{\e\to 0} \p_z^2 V_\e=  \p_z^2 V_0 \,,\quad \text{in $L^\infty_\loc$}.
\end{align}
We have used the Arzela-Ascoli theorem and the uniform $\mathcal C^3$ bound in order to get the convergence up to the second derivative. However, there is no reason why the convergence should hold for the third derivative, due to the lack of compactness.

Looking at \eqref{eq:PUeps}, we see that $\I_\e(U_\e)$ converges uniformly to 1, and, for every given $z\in \R$, \eqref{eq:convegence pzVe} implies that
\begin{align}\label{eq:vareq_V}
U_\e(z)-2U_\e(\bz) + U_\e(0) = V_\e(z)-2V_\e(\bz) + V_\e(0)   \xrightarrow[\e \to 0]{} V_0(z)-2V_0(\bz) + V_0(0).
\end{align}

Passing to the pointwise limit in problem \eqref{eq:PUeps}, we get that $V_0$ solves the following problem:
\begin{equation*}
1 + m(z) = \exp \left( V_0(z)-2 V_0(\bz) + V_0(0) \right). 
\end{equation*}
Then,  we have necessarily:
\begin{equation}\label{eq:V0defsec6}
V_0(z) =   \ds \sum_{k\geq 0} 2^k \log \left( 1 + m( 2^{-k}z ) \right).
\end{equation}
This completes the proof of \Cref{convergence PUeps}(ii), up to the identification of the limit of $\gamma_\e$, if it exists. In our approach, this goes through the characterization of the functional $\Je$ \eqref{eq:def_Jeps}. This was indeed the purpose of \cref{estimate Jeps}. Here comes an important difficulty, as compactness estimates are not sufficient to pass to the limit in $\Je(0,V_\e)$ as $\e\to0$ \eqref{eq:estim_Jeps}, as it would formally involve the pointwise value $\p_z^3V_0(0)$ which is beyond what our compactness estimates can provide. Note that passing to the limit in $\p_g\Je(g,V_\e)$ as $\e\to0$  is not an issue, as it can be encompassed by \eqref{eq:convegence pz2Ve}, see    \eqref{eq:estim_Jeps'}.   

%Furthermore from the assumption of \cref{cond logm inEaz} on $m$ and  \cref{existence sums,invariant subsets} we deduce that $V_0 \in \Eaz$.

It remains to prove that the following limit holds true
\begin{equation}\label{eq:limV3}
\lim_{\e \to 0}\dfrac{1}{4\sqrt{2}\pi} \iint _{\R^2} \exp(-Q (y_1,y_2) ) \left[ y_1^2 \p_z^3 V(\e \tilde{y_1}) +   y_2^2 \p_z^3 V(\e \tilde{y_2}))  \right] dy_1dy_2 = \dfrac12 \p_z^3m(0)\, .
\end{equation}
Indeed, this would directly imply that 
\begin{equation}
\lim_{\e\to 0} \Je(g,V_\e) = - \dfrac{1}2\p_z^3m(0)  + g \p_z^2m(0)  \, , 
\end{equation}
as $\p_z^2 V_0(0) = 2\p_z^2 m(0)$ as a consequence of \eqref{eq:V0defsec6}. We could deduce immediately that the root $\gamma_\e(V_\e)$ converges to the expected value \eqref{eq:gamma0value}. 

In the absence of compactness, we call the contraction argument, in order to prove the following key result:

\begin{lem}\label{lem:last contraction}
For every $\delta >0$, there exists $R_1(\delta)>0$, such that, for every $R\geq R_1(\delta)$, there exists $\e_1(\delta,R)$ such that for all $\e \leq \e_1(\delta,R)$, we have:
\begin{align}\label{eq:convgamma_obj}
   \sup_{ \abs{z} \leq \e R} \abs{ \p_z^3 V_\e(z) - \frac43 \p_z^3 m(z) }   \leq \delta \,.
\end{align}
\end{lem}

%
%
%However things are not so easy, because of the lack of compactness of the balls in $\Eaz$, the $\p_z^3 V_\e$ do not converge necessarily. We will work around that using a contraction argument for the quantity :
%\begin{align}\label{eq:convgamma_obj}
%\ds \sup_{ \abs{z} \leq \e R} \abs{ \p_z^3 V_\e(z) - \beta \p_z^3 m(z) },
%\end{align}
%with $\beta$ well chosen, and any big constant $R \in \R$. 
%The idea is that the objective $V_0$ is known and is third derivative in $0$ can  be expressed in terms of $\p_z^3 m(0)$. 
\begin{proof}
To begin with, we differentiate the problem \eqref{eq:PUeps} three times:
\begin{align*}
\p_z^3 V_\e(z) - \dfrac{1}{4}\p_z^3 V_\e(\bz) = \p_z^3 \log\left (\lambda_\e + m(z)\right ) - \p_z^3 \log\left ( \I_\e(U_\e)(z) \right ).
\end{align*}
We expand the right hand side as usual:
\begin{multline*}
  \p_z^3 V_\e(z) - \dfrac{1}{4} \p_z^3 V_\e(\bz) =  \dfrac{\p^3_z m(z)}{\lambda_\e +m(z)}+\dfrac{3 \p^2_z m(z)\p_z m(z)}{(\lambda_\e +m(z))^2}
+\dfrac{2 \p_z m(z	)^3}{(\lambda_\e +m(z))^3} \\ - W_\e^{(3)}(z)- 3W_\e^{(2)}(z)W_\e^{(1)}(z)-2W_\e^{(1)}(z)^3 .
\end{multline*}
We subtract $\p_z^3 m(z)$ on each side, and we reorganize the terms in order  to conjure the difference $\p_z^3 V_\e(z) - (4/3) \p_z^3 m(z)$ we are interested in:
\begin{multline}\label{eq:convgamma_1} 
\p_z^3 V_\e(z) - \frac43 \p_z^3 m(z) - \dfrac{1}{4}\left( \p_z^3 V_\e(\bz)-\frac43 \p_z^3 m(\bz)  \right) + \dfrac{1}{3} \left(   \p_z^3 m(z) - \p_z^3 m(\bz)\right) = \\
 \p^3_z m(z) \left( \dfrac{1}{\lambda_\e +m(z)}-1 \right) + \dfrac{3 \p^2_z m(z)\p_z m(z)}{(\lambda_\e +m(z))^2}
+\dfrac{2 \p_z m(z	)^3}{(\lambda_\e +m(z))^3}\\  - W_\e^{(3)}(z)- 3W_\e^{(2)}(z)W_\e^{(1)}(z)-2W_\e^{(1)}(z)^3.
\end{multline}
We  estimate below each term of \eqref{eq:convgamma_1}.
First, the terms involving $m$ and its derivatives on the right hand side of \eqref{eq:convgamma_1}   converge to zero, uniformly for $|z|\leq \e R$, as $\e\to 0$, simply because  $m(0) = \p_z m(0)= 0$,  and $\lambda_\e\to 1$.
Actually, the same holds true for the difference of $ \p_z^3 m(z) - \p_z^3 m(\bz)$ by   continuity of   $\p_z^3m$ at the origin. 

Second,  from \cref{unif boundsW}, we know that 
\begin{equation}
\max\left (\left  \| W_\e^{(1)} \right \|_{\infty}, \left  \| W_\e^{(2)} \right \|_{\infty} \right )=  O(\e).
\end{equation}
%
%For every $\delta >0$ there  exists an $\e_m >0$ such that for every $\e \leq \e_m$ :
%\begin{align*}
%\ds \sup_{\abs{z} \leq \e R} \left( \abs{ \dfrac{3 \p^2_z m(z)\p_z m(z)}{(\I_\e(0)+m(z))^2} } 
%+ \abs{\dfrac{2 \p_z m(z	)^3}{(\I_\e(0)+m(z))^3} }  \right) \leq \delta
%\end{align*}
%For the last term in $m$, we use the positivity of $m$ and \cref{estim Ieps}. For every $\delta > 0$ there exists $\e_\delta$ such that for all $\e \leq \e_\delta$ :
%\begin{align*}
%\abs{ \p^3_z m(z) \left( \dfrac{1}{\I_\e(0)+m(z)}-1 \right) } \leq \dfrac{ \nalph{ \log(1+m) }}{1+\delta} \abs{ \I_\e(0)-1 + \sup_{\abs{z} \leq \e R} m(z) }
%\end{align*}
%Again, from \cref{estim Ieps}, for every $\delta$ there exists an $\e_R$ that depends on $R$ and $\delta$ such that for every $\e \leq \e_R$, $ \ds \abs{ \I_\e(0)-1 + \sup_{\abs{z} \leq \e R} m(z) } \leq \delta $.
%Abusing notations, since  the result holds for every positive $\delta$, there exists $\e_R$ such that for all $\e \leq \e_R$ :
%\begin{align*}
%\abs{ \p^3_z m(z) \left( \dfrac{1}{\I_\e(0)+m(z)}-1 \right) } \leq \delta .
%\end{align*}
The remaining term, $W_\e^{(3)}(z)$ is more delicate to handle. In fact, it will result in a contraction estimate, exactly as in \cref{sec properties H}. We recall the expression of $W_\e^{(3)} $ \eqref{eq:W3V}:
\begin{equation*}
W_\e^{(3)}(z) =    
\left\langle \dQVe(z),  \dfrac{1}{4} \D(\p_z^3 V_\e)(z)  + \left( \D(\p_z V_\e)(z)\right)^3  
+  \frac32 \D(\p_z V_\e)(z)  \D(\p_z^2 V_\e)(z)  \right\rangle.
\end{equation*}
As in the proof of \cref{estim Delta 1,estim Delta 2,estim Delta 3}, we get that the last two contributions involving the non-linear and lower order terms $\left( \D(\p_z V_\e) \right)^3$ and  
$ \D(\p_z V_\e)   \D(\p_z^2 V_\e)$ are  $ O(\e)$. 
It remains the  term $\left\langle \dQVe ,  (1/4) \D(\p_z^3 V_\e) \right \rangle$, which is a double integral in variables $(y_1,y_2)$ that we split in two regions of integration:  $\Omega = \left\{ \abs{y_1} \leq R/2, \text{ and } \abs{y_2} \leq R/2 \right\}$ and $  \comega= \left\{ \abs{y_1} > R/2, \text{ or } \abs{y_2} > R/2 \right\}$. 

Let $\delta>0$. We can choose $R_1(\delta)$ large enough so that, for all $R\geq R_1(\delta)$, we have  
\begin{align}
\frac14 \left |\left\langle \dQVe(z)\mathbf{1}_{\comega}(y_1,y_2),   \D(\p_z^3 V_\e)(z) \right\rangle \right | 
\leq \frac12 \left\langle \dQVe(z)\mathbf{1}_{\comega}(y_1,y_2), 1 \right\rangle \nalph{K_0}  \leq \dfrac{\delta}{10}.
\end{align}
%As a consequence, we obtain:
%\begin{align*}
%\abs{\left\langle \dQVe(z)\mathbf{1}_{\Omega_2},  \dfrac{1}{4} \p_z^3 V_\e(\bz) -\dfrac{1}{8} \p^3_z V_\e(\bz+\e y_1)-\dfrac{1}{8} \p^3_z V_\e(\bz+\e y_2) \right\rangle } \leq \left\langle \dQVe(z)\mathbf{1}_{\abs{y_1} \geq  R/2 ,  \abs{y_2} \geq  R/2}, \dfrac{\nalph{V_\e}}{2}\right\rangle.
%\end{align*}
%But $\nalph{V_\e} \leq \nalph{K_0}$, and therefore if we abuse notation by denoting again $\delta$ the small constant :
%\begin{align*}
%\abs{\left\langle \dQVe(z) \mathbf{1}_{\Omega_2},  \dfrac{1}{4} \p_z^3 V_\e(\bz) -\dfrac{1}{8} \p^3_z V_\e(\bz+\e y_1)-\dfrac{1}{8} \p^3_z V_\e(\bz+\e y_2) \right\rangle } \leq \delta.
%\end{align*}
In the region where $y_1$ and $y_2$ are both below $R/2$,  
we introduce the difference with $\p_z^3m$, as in \eqref{eq:convgamma_obj}:
\begin{equation*}
\left\langle \dQVe(z) \mathbf{1}_{\Omega},  \dfrac{1}{4} \p_z^3 V_\e(\bz) -\dfrac{1}{8} \p^3_z V_\e(\bz+\e y_1)-\dfrac{1}{8} \p^3_z V_\e(\bz+\e y_2) \right\rangle  = A + B,
\end{equation*} 
where 
\begin{multline*}
A = \left\langle \dQVe(z)\mathbf{1}_{\Omega},  \dfrac{1}{4} \left( \p_z^3 V_\e(\bz)-\frac43 \p_z^3  m(\bz) \right)   \right.
\\
\left. -\dfrac{1}{8} \left( \p^3_z V_\e(\bz+\e y_1)-\frac43 \p_z^3 m(\bz + \e y_1 ) \right)-\dfrac{1}{8} \left(\p^3_z V_\e(\bz+\e y_2) -\frac43 \p_z^3 m(\bz +\e y_2)\right) \right\rangle  
\end{multline*}
and
\begin{equation*}
 B = \frac16 \left\langle \dQVe(z)\mathbf{1}_{\Omega},    \left ( \p^3_z m(\bz) - \p_z^3 m(\bz + \e y_1 )\right ) +\left ( \p^3_z m(\bz)  - \p_z^3 m(\bz +\e y_2) \right )\right\rangle.
\end{equation*}
By construction, we have $|\bz + \e y_i| \leq \e R/2 + \e R/2 \leq \e R$. Therefore, we have
\begin{align*}
\abs{A} & \leq \left\langle \dQVe(z) \mathbf{1}_{\Omega}, \dfrac{1}{4} \underset{ \abs{z} \leq \e R  }  {\sup} \abs{ \p_z^3 V_\e(z)- \frac43 \p_z^3 m(z)} + \dfrac{2}{8} \sup_{\abs{z} \leq \e R} \abs{ \p_z^3 V_\e(z)- \frac43 \p_z^3 m(z)} \right\rangle \\
& \leq \dfrac{1}{2}\sup_{\abs{z} \leq \e R} \abs{ \p_z^3 V_\e(z)- \frac43 \p_z^3 m(z)}.
\end{align*}
As for $B$ we find:
\begin{equation*}
\abs{B}  \leq \dfrac13 \left\langle \dQVe(z) \mathbf{1}_{\Omega} ,   \underset{\abs{z} \leq \e R} { \text{ osc }}(\p_z^3 m)  \right\rangle  \leq  \dfrac13 \underset{\abs{z} \leq \e R} { \text{ osc }}\left (\p_z^3 m\right ) \xrightarrow[\e \to 0]{} 0 . 
\end{equation*}
Going back to \eqref{eq:convgamma_1}, we have shown that for $R\geq R_1$, there exists $\eps_1>0$ small enough such that for all $\e \leq \e_1$ we have:
\begin{align*}
\ds \sup_{\abs{z} \leq \e R}   \abs{ \p_z^3 V_\e(z)- \frac43 \p_z^3 m(z) } 
& \leq \ds \frac\delta4   + \dfrac{1}{4}  \sup_{\abs{z} \leq \e R} \abs{ \p_z^3 V_\e(\bz)- \frac43 \p_z^3 m(\bz)} +    
 \dfrac{1}{2}\sup_{\abs{z} \leq \e R} \abs{ \p_z^3 V_\e(z)- \frac43 \p_z^3 m(z)}   \\
& \leq  \frac\delta4 +   \dfrac{3}{4}  \sup_{\abs{z} \leq \e R} \abs{ \p_z^3 V_\e(z)- \frac43 \p_z^3 m(z)} .
\end{align*}
As a consequence, we find that 
\begin{align*}
\sup_{\abs{z} \leq \e R}   \abs{ \p_z^3 V_\e(z)- \frac43 \p_z^3 m(z) } \leq    \delta .
\end{align*}
This completes the proof of \cref{lem:last contraction}.
\end{proof}

Back to \eqref{eq:limV3}, we recall that $\abs{\tilde{y_i}} \leq \abs{ y_i}+1$, as a by-product of Taylor expansions. Let $\delta>0$, and take $R$ sufficiently large such that 
\begin{equation}
\dfrac{1}{4\sqrt{2}\pi}\iint _{\comega}  \exp(-Q (y_1,y_2) ) \left (  \nalph{K_0} + \dfrac12 \p_z^3m(0)\right ) \left[ y_1^2  +   y_2^2   \right] dy_1dy_2  
\leq  \frac\delta{10}\,,
\end{equation}
where $\Omega = \left\{ \abs{y_1} \leq R - 1, \text{ and } \abs{y_2} \leq  R - 1 \right\}$. The other part of the double integral is:
\begin{equation}
\dfrac{1}{4\sqrt{2}\pi}\iint _{\Omega} \exp(-Q (y_1,y_2) ) \left[ y_1^2 \p_z^3 V(\e \tilde{y_1}) +   y_2^2 \p_z^3 V(\e \tilde{y_2}))  \right] dy_1dy_2 .
\end{equation}
Using \cref{lem:last contraction} and the continuity of $\p_z^3m$ at $z = 0$, we can find $\e_1>0$ such that for all $\e\leq \e_1$, 
\begin{multline*}
\abs{ \dfrac{1}{4\sqrt{2}\pi}\iint _{\Omega} \exp(-Q (y_1,y_2) ) \left[ y_1^2 \left ( \p_z^3 V(\e \tilde{y_1})- \dfrac43 \p_z^3m(0)\right ) +   y_2^2 \left ( \p_z^3 V(\e \tilde{y_2})- \dfrac43 \p_z^3m(0)\right )  \right] dy_1dy_2    } \\
\leq \left ( \dfrac{1}{4\sqrt{2}\pi}\iint _{\Omega} \exp(-Q (y_1,y_2) ) \left[ y_1^2  +   y_2^2   \right] dy_1dy_2\right ) \dfrac{\delta}{10}. 
\end{multline*}
Putting all the pieces together, and using that $\frac1{\sqrt{2}\pi}   \iint_{\R^2} \exp(-Q (y_1,y_2) ) \left[ y_1^2  +   y_2^2   \right] dy_1dy_2 = \frac32$ \eqref{eq:covariance}, we deduce that:
\begin{equation}
\abs{\dfrac{1}{4\sqrt{2}\pi} \iint _{\R^2} \exp(-Q (y_1,y_2) ) \left[ y_1^2 \p_z^3 V(\e \tilde{y_1}) +   y_2^2 \p_z^3 V(\e \tilde{y_2}))  \right] dy_1dy_2 - \dfrac12 \p_z^3m(0)} \leq \delta \, .
\end{equation}
Hence,  the limit announced in \eqref{eq:limV3} holds true. This completes the proof of the asymptotic behavior $(\lep, U_\e)\to (\lambda_0,U_0)$ as described in \Cref{convergence PUeps}(ii).

\section{Extension to higher dimensions}\label{sec large dim}

Our methodology can be extended to higher dimension, without too much effort. 
This section is devoted to the  generalization of  the elements of proof that were specific to the one-dimensional case. 
%For instance, one can still prescribe the condition 
%\begin{align*}
%\log (1+m) \in \Eaz,
%\end{align*}
%but for clarity we will not write down precisely the successive derivations of the $\log$. 

All the estimates  on the operator $\H$ and its constitutive pieces are still operational in   higher dimension.  The only part of our proof that requires some specific attention is the construction of the linear part $\gamma_\epsilon(V_\e)$ which was performed in~\Cref{sec existunique gammaeps}.  Indeed, we used a monotonicity argument to show that $\gamma_\epsilon(V_\e)$ can be defined in a unique way. 

We proceed as in~\Cref{sec existunique gammaeps}.
First we show formally how to obtain the expression of the vector $\gamma_0$ \eqref{eq:gamma0value} via suitable Taylor expansions. Then, we justify these Taylor expansions, and we exhibit a monotonic function that enables to conclude, exactly as in dimension 1.

\subsection{The formal expression of the linear part $\gamma_0$}
Following the very same heuristics as in \cref{sec heuristicsgammaeps}, but being careful during the Taylor expansions, we formally end up with the following matrix valued  identity:
\begin{multline}\label{eq:impgamma_multid}
D^2 V(0)   \left( \dfrac{1}{(\sqrt{2}\pi)^d} \ds \iint_{\R^{2d}} e^{ -Q(y_1,y_2)} ( y_1 \otimes y_1  + y_1 \otimes y_2 ) dy_1 dy_2  \right)  \gamma_0  \\
=\dfrac{1}{2} D^3 V(0)   \left ( \dfrac{1}{(\sqrt{2}\pi)^d}  \iint_{\R^{2d}} e^{  -Q(y_1,y_2) } y_1 \otimes y_1 dy_1 dy_2 \right ).
\end{multline}
The quadratic form $Q$ yields the multivariate centered gaussian distribution associated with the following covariance matrix $\Sigma \in \mathcal M_{2d}(\R) $:
\begin{align*}
\Sigma = \dfrac{1}{4} \begin{pmatrix}
3 \Id & - \Id  \\
- \Id & 3 \Id
\end{pmatrix}.
\end{align*}
The Kronecker product $y_1 \otimes y_1 $ yields a  matrix of moments, and so %Since $\mathbf{Q}$ can be easily inverted :
%\begin{align*}
%\mathbf{Q}^{-1}  =\dfrac{1}{2} \left( \begin{array}{cc}
%3 \Id &  \Id  \\
% \Id & 3 \Id \end{array}\right),
%\end{align*}
the relation \eqref{eq:impgamma_multid} can be simplified, similarly to the one dimensional case, so as to obtain: 
\begin{align*}
\left ( D^2V(0)  \left(\dfrac{3}{4}-\dfrac{1}{4}\right) \Id\right ) \gamma_0  &= \frac12 D^3V(0) \dfrac{3}{4} \Id, \\
\frac12 D^2V(0)\gamma_0 & = \frac38 D^3V(0)  \Id . 
\end{align*}
The righ hand side is a tensor applied to a matrix yields a vector that can be simplified even further using tensorial properties:
$D^3 V(0)  \Id =  D ( \Delta V)(0)$.
Then, provided that $D^2V(0)$ is non degenerate,  we obtain the limited expected value of $\gamma_0$ in dimension higher than $1$, that is a generalization of \eqref{eq:gammaepsd=1}:
\begin{align*}
\gamma_0(V) = \dfrac{3}{4} \left ( D^2 V(0)\right )^{-1}  D ( \Delta V)(0).
\end{align*}
In the case where $V_0$ is given by \eqref{eq:gamma0value} through the fixed point procedure, we obtain
\begin{equation}
\gamma_0(V_0) = \dfrac{1}{2} \left ( D^2 m(0)\right )^{-1}  D ( \Delta m)(0).
\end{equation}

\subsection{Extension of the proof of~\cref{existunique gammaeps} (\cref{sec:proof of prop24})}
We now fix $V \in K$, where $K$ is a  ball of $\Eaz$. The purpose is to prove that there is a unique solution in $\R^d$ of the following problem:
\begin{align}\label{eq:Je_0}
\Je(\gamma,V) = 0.
\end{align}
We insist upon the fact that the variable $\gamma$ belongs to $\R^d$ and the function $\Je(\cdot,V)$ is now defined as a vector field on $\R^d$, $\Je : \R^d \times \Ea \to \R^d$.

As in \cref{sec:proof of prop24}, we can obtain the following estimate
\begin{equation}\label{eq:prep brouwer}
\Je(g,V) = \Je(0,V) + \frac12 D^2  V(0)  g + O(\e)\, , 
\end{equation}
by means of refined Taylor expansions, where $\Je(0,V)$ is bounded {\em a priori}, independently upon $\e>0$ for $V\in K$. To prove the existence of a root $\gamma_\e$, we used the mean value theorem in the proof of~\cref{existunique gammaeps}. The analogous statement in higher dimension is the Brouwer fixed point theorem. Indeed, \eqref{eq:Je_0} can be recast as follows:
\begin{align*}
g = \left ( \Id +  \dfrac{1}{2} D^2 V(0) \right )^{-1}\left (  g -  \Je(0,V) + O(\e)\right ) = \mathcal{T}_\e(g) .
\end{align*}
Thus, we are led to finding a fixed point of a continuous function. As in the one-dimensional case, thanks to the lower bounded $D^2V(0) \geq \mu_0 \Id$ encoded in the definition of $\Eaz$ \eqref{eq:E0}, we can show easily that there exists $R_K$ such that the ball of radius $R_K$ in $\R^d$ is left invariant by $\mathcal{T}$. 
%We use the decomposition $Id + \frac12 D^2 (V-m)(0)  + \frac12 D^2 m(0)$. Then the argument is  that $D^2V(0)$ is a uniformly positive definite matrix, and $\mathcal{R}_\e$ and $O_\e(g)$ are uniformly bounded. 
Brouwer's fixed point theorem guarantees that there  exists a fixed point $\gamma_\e$ to $\mathcal{T}$, which is also a root of \eqref{eq:Je_0}. 

For the uniqueness part, we can use strict monotonicity, similarly as in the one dimensional case. This is possible, thanks to \eqref{eq:estim_Jeps'} :
\begin{align}\label{eq:pJe_approx}
D_g \Je(g,V)& =\frac12  D^2 V(0)  + O(\e).
\end{align}
We deduce from this strong estimate that the vector field $\Je(\cdot,V)$ is locally uniformly monotonic, in the sense that there exists $\mu_K$ such that the following inequality holds true for  all  $\e$ sufficient small, and every  $g_1,g_2\in B(0,R_K)$: 
\begin{align}\label{eq:cond posvectorfield}
% \left\langle F(X)-F(Y) ; X-Y \right\rangle \geq \delta \norm{X-Y}^2.
\left(\Je(g_1,V)-\Je(g_2,V)\right) \cdot( g_1-g_2 )\geq \frac12 \mu_K \norm{g_1-g_2}^2.
\end{align}
This monotonicity condition is clearly satisfied, as it is equivalent to the following first order condition,
\begin{align}\label{eq:jacob pos}
\dfrac{1}{2} \left( D_g \Je(g,V) + D_g \Je(g,V)^\top \right) \geq \mu_K Id,
\end{align}
It is immediate that any strictly monotonic vector field admits at most one root. This completes the proof of uniqueness of $\gamma_\e(V)$.

\bibliographystyle{apalike}
\bibliography{References2}
\bigskip
\bigskip
\bigskip

\end{document}